\documentclass[11pt]{article}

\usepackage{amsmath, amssymb, amsfonts}   
\usepackage{mathtools,float, keyval}  
\usepackage{commath}              
\usepackage{graphicx}
\usepackage[margin=2cm]{geometry}
\usepackage{amsmath,amsfonts,amssymb,amsthm}
\usepackage{latexsym}
\usepackage{setspace}
\usepackage{caption}

\usepackage{multirow}
\usepackage{subcaption}
\usepackage{algorithmic}
\usepackage{algorithm}
\usepackage{graphicx}
\usepackage{float}
\usepackage{wrapfig}
\usepackage{adjustbox}
\usepackage{hyperref}                 
\usepackage{placeins}
\usepackage{textcomp}
\usepackage{gensymb}
\usepackage[title]{appendix}
\usepackage{authblk}

\newtheorem{theorem}{Theorem}[section]
\newtheorem{lemma}[theorem]{Lemma}

\newtheorem{proposition}[theorem]{Proposition}

\newtheorem{remark}{Remark}[section]

\DeclareMathOperator{\vect}{vec}

\DeclareMathOperator{\diag}{diag}

\begin{document}

\title{Symmetric rank-one updates from partial spectrum with an application to out-of-sample extension}

\author[1]{Roy Mitz\thanks{roymitz@mail.tau.ac.il} }
\author[1]{Nir Sharon\thanks{nsharon@tauex.tau.ac.il} }
\author[1]{Yoel Shkolnisky\thanks{yoelsh@tauex.tau.ac.il}}
\affil[1]{\footnotesize{School of Mathematical Sciences, Tel-Aviv University, Tel-Aviv, Israel}}
\date{}

\maketitle
\begin{abstract}
Rank-one update of the spectrum of a matrix is a fundamental problem in classical perturbation theory. In this paper, we consider its variant where only part of the spectrum is known. We address this variant using an efficient scheme for updating the known eigenpairs with guaranteed error bounds. Then, we apply our scheme to the extension of the top eigenvectors of the graph Laplacian to a new data sample. In particular, we model this extension as a perturbation problem and show how to solve it using our rank-one updating scheme. We provide a theoretical analysis of this extension method, and back it up with numerical results that illustrate its advantages.
\end{abstract}

\section{Introduction}

The last few decades have witnessed the emergence of various algorithms that require the calculation of the eigendecomposition of matrices. A few known examples are Google's PageRank \cite{page1999pagerank}, PCA \cite{shlens2014tutorial}, Laplacian eigenmaps \cite{belkin2003laplacian}, LLE \cite{roweis2000nonlinear} and MDS \cite{buja2008data}. Since datasets nowadays may contain tens of millions of data points, an efficient calculation of the eigendecomposition becomes fundamental. In many scenarios, only part of the eigendecomposition, i.e., only the leading eigenvalues and eigenvectors, can or need to be calculated. While algorithms for eigendecomposition, such as the Lanczos algorithm and some variants of SVD, are designed especially for this task, they still require a hefty amount of calculations. A natural question that arises in such cases is how to update the eigendecomposition of a matrix given its partial eigendecomposition and some ``small" perturbation to it, without repeating the entire decomposition again.

In this paper, we focus on rank-one updates of symmetric matrices. The classical approach for such an update is updating the eigenvalues using the roots of the secular equation, see e.g., \cite{bunch1978rank}. However, several other approaches for updating the eigenvalues and eigenvectors of a perturbed matrix have been suggested. The popular ones are quite general and include recalculating from scratch or restarting the power method \cite{langville2006updating} and perturbation methods \cite{stewart1990matrix}. Some methods that utilize the structure of a specific problem were suggested, with Google's page rank being the most popular application \cite[Chapter~10]{langville2011google}. Another important method is based on a geometric embedding of the available data \cite{brand2006fast}. This approach becomes computationally attractive when one updates a low-rank matrix. 

Many of the methods mentioned above are inapplicable or provide very poor guarantees in cases where we do not have access to the complete eigendecomposition. Some methods assume that the updated matrix is low-rank, which is not always the right model for real-world data. Finally, almost none of the existing approaches is equipped with  error analysis. In our method, we provide a rank-one update algorithm that does not require the full eigendecomposition of the matrix and does not assume that it is low rank. We demonstrate that the structure of the problem enables us to use the unknown tail of the eigenvalues in order to improve the accuracy of the update. Additionally, the complexity of our algorithm is linear in the number of rows of the matrix. We also analyse the accuracy of our method, showing that it is independent of the number of unknown eigenpairs, but rather only depends on their ``behavior". This observation  is confirmed by both synthetic and real-world examples.

The eigenvalues and eigenvectors of the graph Laplacian have been of special interest recently, as evident by its various applications in machine learning, dimensionality reduction \cite{belkin2003laplacian, coifman2006diffusion}, clustering \cite{ng2002spectral, von2007tutorial}, graph theory \cite{chung1997spectral} and image processing \cite{coifman2008graph}. The problem of out-of-sample extension of the graph Laplacian, which will be described in detail later, is essentially updating the eigendecomposition of the graph Laplacian after the insertion of a new vertex to the graph. This problem is often addressed by the Nystr{\"o}m method \cite{bengio2004out}. We propose a different approach to the extension, based on the observation that under mild assumptions, the insertion of a new vertex to a graph translates to an almost rank-one update of the corresponding graph Laplacian matrix. Then, we apply our rank-one update algorithm to estimate the extension with high accuracy.

The paper is organized as follows. In Section~\ref{sec:rank_one_update}, we derive our algorithm for the symmetric rank-one update based on partial eigendecomposition and analyse its error. In Section~\ref{sec:updating_problem}, we describe the application of the algorithm to the extension of the graph Laplacian. In Section~\ref{sec:numeric}, we illustrate numerically some of our theoretical results from Section~\ref{sec:rank_one_update} and Section~\ref{sec:updating_problem} for both synthetic and real data. We give some concluding remarks in Section~\ref{sec:conclusions}.

\section{rank-one update with partial spectral information}  \label{sec:rank_one_update}

Computing the spectrum of a matrix following its rank-one update is a classical task in perturbation theory, e.g., \cite[Chapter 7]{bender2013advanced}. Given a rank-one perturbation to a matrix, the spectrum of the perturbed matrix is related to that of the original matrix via the secular equation, which involves the entire spectrum of the original matrix. However, if only the few leading eigenpairs of the original matrix are known, the classical approach requires further adaptation. Inspired by \cite{bunch1978rank}, we propose a solution to the ``partial knowledge" rank-one update problem, where we aim to estimate the leading eigenpairs of a matrix after a rank-one perturbation, having only part of the eigendecomposition of the original matrix. We describe in detail the derivation of our method and provide error bounds and complexity analysis. 

\subsection{Notation, classical setting, and problem formulation} \label{subsec:notation_rank_one_update}
We denote by $X = [x_1x_2 \cdots x_n]$ a matrix expressed by its column vectors, and by $X^{(m)} = [x_1\cdots x_m]$ its truncated version consisting only of its first $m$ columns, $m<n$. Let $A$ be an $n \times n$ symmetric real matrix with real (not necessarily distinct) eigenvalues $\lambda_1 \ge \lambda_2 \ge \cdots \ge \lambda_n$ and their associated orthogonal eigenvectors $q_1,\ldots,q_n$. We denote this eigendecomposition by $A = Q\Lambda Q^T$, with $Q = [q_1q_2 \cdots q_n]$ and $\Lambda = \diag (\lambda_1, \ldots,\lambda_n)$. We focus on the problem of (symmetric) rank-one update, where we wish to find the eigendecomposition of 
\begin{equation} \label{eq:prob}
A + \rho vv^T, \quad \rho \in \mathbb{R}, \quad  v \in \mathbb{R}^n,  \quad  \norm{v} = 1.
\end{equation}
We denote the updated eigenvalues by $t_1 \ge t_2 \ge \cdots \ge t_n$ and their associated, orthogonal eigenvectors by $p_1,\ldots,p_n$ to form the decomposition $A + \rho vv^T = PTP^T$, with $P = [p_1p_2 \cdots p_n]$ and $T = \diag (t_1, \ldots,t_n)$. Approximated objects (whether scalars or vectors) constructed in this section are denoted by an over tilde. For example, an approximation for $x$ is denoted by $\widetilde{x}$.

The relation between the decompositions before and after the rank-one update is well-studied, e.g., \cite{bunch1978rank, ding2007eigenvalues}. Without loss of generality, we further assume that $\lambda_1 > \lambda_2 > \cdots > \lambda_n$ and that for $z = Q^Tv$ we have $z_j \neq 0$ for all $1 \leq j \leq n$. The deflation process in~\cite{bunch1978rank} reduces any update~\eqref{eq:prob} to this form.  Given the eigendecomposition $A = Q\Lambda Q^T$, the updated eigenvalues $t_1,...,t_n$ of $A + \rho v v^T$ are given by the $n$ roots of the secular equation
\begin{equation} \label{eqn:secular_equation}
w(t) = 1 + \rho \sum_{i=1}^{n}\frac{z_i^2}{\lambda_i - t} , \quad  z = Q^Tv.
\end{equation}
The corresponding eigenvector for the $k$-th root (eigenvalue) $t_k$ is given by the explicit formula
\begin{equation}  \label{eqn:EigenvaectorFormula}
p_k = \frac{Q\Delta_k^{-1}z}{\norm{Q\Delta_k^{-1}z}} , \quad  z = Q^Tv , \quad  \Delta_k = \Lambda - t_k I .
\end{equation}

An important assumption in the above is the knowledge of the full eigendecomposition of the matrix $A$. This is not always feasible in modern problems due to high computational and storage costs. Therefore, a natural question is what can one do in cases where only part of the spectrum is known. Thus, we are interested in the following problem. Let $A$ be an ${n \times n}$ real symmetric matrix and let $1\le m < n$. Assume we have only the first $m$ leading eigenvalues $\lambda_1 > \lambda_2 > \cdots > \lambda_m$ of $A$ and their associated eigenvectors $q_1,\ldots,q_m$. Find an estimate to the first $m$ leading eigenpairs of $A + \rho vv^T$ with $\rho \in \mathbb{R}$ and $\norm{v}=1$.

\subsection{Truncating the secular equation} \label{sec:trunc_secular}

We start by considering the first part of the above problem --- the eigenvalues. The classical perturbation method solves for the eigenvalues of $A + \rho vv^T$ by finding the roots of the secular equation~\eqref{eqn:secular_equation}. We introduce two modifications of the secular equation, adapted to our new setting.

Using the notation of~\eqref{eqn:secular_equation}, we have from the orthogonality of $Q$ that 
\begin{equation}
\|z\|^2 = \|Q^Tv\|^2 = \|v\|^2 = 1.
\end{equation}
Therefore, $\sum_{j=m+1}^{n}{z_j^2} = 1 - \sum_{i=1}^{m}{z_i^2}$. Since the last $n-m$ eigenvalues of $A$ are unknown, we denote by  $\mu < \lambda_m$ a fixed scalar, whose purpose is to approximate $\lambda_j$, $j=m+1,\ldots,n$. Choosing $\mu$ will be discussed below. We then define the first order truncated secular equation by
\begin{equation} \label{eq:TSE}
w_{1}(t ; \mu) = 1 + \rho \left( \sum_{i=1}^{m} {\frac{z_i^2}{\lambda_i - t}} + \frac{1 - \sum_{i=1}^{m}{z_i^2}}{\mu - t} \right) ,
\end{equation}
where $z=(Q^{(m)})^Tv$ is a vector of length $m$ (the first $m$ entries of $Q^Tv$), with the columns of the matrix $Q^{(m)}$ consisting of the (orthogonal) eigenvectors corresponding to the (known) leading eigenvalues of $A$.

As a first observation, we bound the error obtained from the new formula. Namely, we show that the deviation of the $m$ largest roots of the truncated secular equation \eqref{eq:TSE} from the roots $t_1,\ldots,t_m$ of~\eqref{eqn:secular_equation} is of the order of $\max_{m+1 \leq j \leq n} \abs{\lambda_j - \mu }$.
\begin{proposition} \label{prop:TSE_roots}
Let $\rho>0$. Then, there exist $m$ roots $\widetilde{t}_1,\ldots,\widetilde{t}_m$ of $w_{1}(t ; \mu)$ of \eqref{eq:TSE}, such that
\begin{equation} \label{eq:error_first_order}
\abs{ t_k-\widetilde{t_k}} \le C_k \max_{m+1 \leq j \leq n} \abs{\lambda_j - \mu}  , \quad k=1,\ldots,m,
\end{equation}
where ${t_1},\ldots,{t_m}$ are the largest $m$ roots of \eqref{eqn:secular_equation}, and $C_k$ is a constant bounded by 
\[ (\lambda_{k} - \mu) ^{-1}(\lambda_m - \lambda_{m+1})^{-1} \max \{ (\lambda_{k} - \lambda_1)^2,(\lambda_{k-1} - \lambda_{n})^2 \}. \]
\end{proposition}
\begin{proof}
We start by showing the existence of the roots of \eqref{eq:TSE}. Indeed, since $w_1(t ; \mu)$ is monotone and since $\lim_{t\to\lambda_i^{\pm}} w_1(t;\mu) = \mp\infty$ for $i=1,\ldots,m$, there exists a single root in any segment $[\lambda_j,\lambda_{j-1}]$, $2 \le j \le m$. Additionally, since $\lim_{t \rightarrow \infty}w_1(t ; \mu) = 1 $ and by classical perturbations bounds for the eigenvalues of symmetric matrices~\cite[Corollary 8.1.6]{golub2012matrix}, there exits a single root in the interval $[\lambda_1, \lambda_1 + \rho]$. These are the $m$ roots of $w_1(t, \mu)$ in the interval $[\lambda_m, \lambda_1 + \rho]$.

For the bound in~\eqref{eq:error_first_order}, we expand
\begin{equation} \label{eqn:summand_first_app}
  \frac{z_i^2}{\lambda_i - t} = \frac{z_i^2}{\mu-t} +  \frac{z_i^2(\mu - \lambda_i)}{(\mu-t)(\lambda_i - t)} .
\end{equation}
Therefore, by splitting the sum in \eqref{eqn:secular_equation} and using \eqref{eqn:summand_first_app} we have
\begin{equation} \label{eq:se_split}
w(t)  =  \underbrace{1 + \rho \sum_{i=1}^{m} {\frac{z_i^2}{\lambda_i - t}} + \rho \sum_{i=m+1}^{n}  {\frac{z_i^2}{\mu -t}}}_{w_{1}(t ; \mu)}  + \underbrace{\rho \sum_{i=m+1}^{n} \frac{z_i^2 (\mu - \lambda_i)}{(\mu - t)(\lambda_i-t)}}_{e(t;\mu)}  . 
\end{equation}
By Taylor expansion of \eqref{eqn:secular_equation},
\begin{equation}
0  = w(\widetilde{t}_k - (\widetilde{t}_k-t_k)) = w(\widetilde{t}_k) - (\widetilde{t}_k-t_k) \frac{d}{dt}w(\xi) =  w_{1}(\widetilde{t}_k ; \mu) + e(\widetilde{t}_k ; \mu ) - (\widetilde{t}_k-t_k) \frac{d}{dt}w(\xi) ,
\end{equation}
where $\widetilde{t}_k, \xi \in (\lambda_{k},\lambda_{k-1})$.

By definition, $w_{1}(\widetilde{t}_k ; \mu)=0$. In addition, the derivative of the secular equation does not have a real root, meaning that   
\begin{equation} \label{eqn:bnd_roots_differ}
    (\widetilde{t}_k-t_k) = \frac{e(\widetilde{t}_k;\mu)}{ \frac{d}{dt}w(\xi)} .   
\end{equation}

For the error term $e(\widetilde{t}_k;\mu)$ we have
\begin{equation} \label{eq:min1}
\abs{e(\widetilde{t}_k;\mu)} \leq \frac{\rho}{\abs{\mu - \widetilde{t}_k}\abs{\lambda_{m+1}-\widetilde{t}_k} }\sum_{i=m+1}^{n} z_i^2 \abs{\lambda_i - \mu} \leq \frac{\rho}{\abs{\mu - \widetilde{t}_k}\abs{\lambda_{m+1}-\widetilde{t}_k} } \max_{m+1 \leq i \leq n}\abs{\lambda_i - \mu} ,
\end{equation}
where the last inequality follows since $\|z\|_2 = 1$. In addition, $\widetilde{t}_k\ge\lambda_{m}$ and $\mu < \lambda_{m}$ so the denominator in~\eqref{eq:min1} is bounded from below by $|\mu - \lambda_{k}||\lambda_m - \lambda_{m+1}|$. Back to \eqref{eqn:bnd_roots_differ}, the derivative of the secular equation is
\begin{equation}
\frac{d}{dt}w(t) =  \rho \sum_{i=1}^n \frac{z_i^2}{(\lambda_i-t)^2} ,
\end{equation}
and thus
\begin{equation}
\abs{ \frac{d}{dt}w(t)} \ge \rho \sum_{i=1}^n \frac{z_i^2}{(\lambda_i-t)^2} \ge  \rho \min_{1\le i\le n}  \left\lbrace \frac{1}{(\lambda_i-t)^2}\right\rbrace  \sum_{i=1}^n z_i^2 = \rho \left\lbrace \max_{1\le i\le n}(\lambda_i-t)^2 \right\rbrace^{-1} .
\end{equation}
Therefore, 
\begin{equation}
\abs{\frac{1}{ \frac{d}{dt}w(t)}} \le \frac{\max \{ (\lambda_{k} - \lambda_1)^2,(\lambda_{k-1} - \lambda_{n})^2 \} }{\rho  } , \quad t\in ( \lambda_{k},\lambda_{k-1} ).
\end{equation}
\end{proof}
\begin{remark}
Proposition~\ref{prop:TSE_roots} describes the case of $\rho>0$. The case of $\rho<0$ is analogous with one exception -- the last root $\widetilde{t}_m$ is merely guaranteed to lie in the segment $[\mu,\lambda_{m}]$.
Consequently, the constant $C_m$ cannot be bounded with the same arguments.
\end{remark}

We now address the issue of choosing $\mu$. A common assumption in many real world applications is that the matrix $A$ is low-rank, and thus the unknown eigenvalues are zero, implying the choice $\mu = 0$. This is indeed the case for several important kernel matrices, as we will see in the next section. For matrices that are not low rank, the error term of Proposition~\ref{prop:TSE_roots} using $\mu = 0$ would be $O(|\lambda_{m+1}|)$ and we have no reason to believe that this will result in a good approximation. 

A better method for choosing $\mu$ would be to minimize the sum in the middle term of \eqref{eq:min1}. However, an analytic minimizer is not attainable in this case since both $\lambda_i$ and $z_i^2$, $i = m+1,...,n$, are unknown. Shortly, we will devise an approximation of the secular equation for which an analytic minimizer can be calculated. Nevertheless, assuming the trace of $A$ is available, an intuitive choice for $\mu$ that works well in practice and is also fast to compute is the mean of the unknown eigenvalues, which is accessible since 
\begin{equation} \label{eq:mu_mean}
\mu_{mean} = \frac{\sum_{i=m+1}^n \lambda_i}{n-m} = \frac{\operatorname{tr}(A) - \sum_{i=1}^{m} \lambda_i}{n - m} .
\end{equation}

Following the proof of Proposition~\ref{prop:TSE_roots}, we are encouraged to try to improve the approximation to the eigenvalues of~\eqref{eq:prob} by using a higher order approximation for \eqref{eqn:summand_first_app}, namely,
\begin{equation} \label{eqn:ImprovedExpansion}
  \frac{z_i^2}{\lambda_i - t} = \frac{z_i^2}{\mu-t} -  \frac{z_i^2(\lambda_i - \mu)}{(\mu-t)^2} + \frac{z_i^2(\lambda_i - \mu)^2}{(\mu - t)^2(\lambda_i - t)} .
\end{equation}
Since $Aq_i = \lambda_i q_i$, we have
\begin{equation}
\sum_{i=m+1}^{n}z_i^2\lambda_i = \sum_{i=m+1}^{n}(q_i^Tv)(q_i^Tv)\lambda_i = \sum_{i=m+1}^{n}(v^T\lambda_iq_i)(q_i^Tv) = \sum_{i=m+1}^{n}(v^TAq_i)(q_i^Tv) ,
\end{equation}
and thus
\begin{equation} \label{eq:s}
\sum_{i=m+1}^{n}z_i^2\lambda_i = v^TA \left( I - Q^{(m)} (Q^{(m)})^T \right) v \triangleq s ,
\end{equation}
which is a known quantity. This analysis gives rise to the second order approximation of the secular equation
\begin{equation} \label{eq:CTSE}
w_{2}(t ; \mu) = 1 + \rho \left( \sum_{i=1}^{m} {\frac{z_i^2}{\lambda_i - t}} +  \frac{1 - \sum_{i=1}^{m}{z_i^2}}{\mu-t} - \frac{s - \mu (1 - \sum_{i=1}^{m}{z_i^2})}{(\mu - t)^2} \right) .
\end{equation}
In this case, the roots of $w_{2}(t ; \mu)$ of \eqref{eq:CTSE} are at most $\max_{m+1 \leq i \leq n} (\lambda_i -\mu)^2 $ away from the roots of the original secular equation \eqref{eqn:secular_equation}. This is concluded in the next result, which is analogous to Proposition~\ref{prop:TSE_roots}.

\begin{proposition}  \label{prop:CTSE_roots}
Let $\rho>0$. Then, there are $m$ roots $\widetilde{t}_1,\ldots,\widetilde{t}_m$ of $w_{2}(t ; \mu)$ of \eqref{eq:CTSE}, such that
\begin{equation} \label{eq:tse_second_order}
\abs{ t_k-\widetilde{t}_k} \le C_k \max_{m+1 \leq j \leq n} (\lambda_j - \mu)^2  , \quad k=1,\ldots,m,
\end{equation}
where ${t_1},\ldots,{t_m}$ are the largest $m$ roots of \eqref{eqn:secular_equation}, and $C_k$ is a constant bounded by 
\[ (\lambda_{k} - \mu) ^{-2}(\lambda_m - \lambda_{m+1})^{-1} \max \{ (\lambda_{k} - \lambda_1)^2,(\lambda_{k-1} - \lambda_{n})^2 \} . \]
\end{proposition}
\begin{proof}
This proof is similar to the proof of Proposition~\ref{prop:TSE_roots}. Here, we have $w(t) = w_{2}(t;\mu) + e(t;\mu)$ with 
\begin{equation} \label{eq:err2}
e(t;\mu) = \rho \sum_{i=m+1}^{n} \frac{z_i^2 (\lambda_i - \mu)^2}{(\mu - t)^2(\lambda_i-t)} .
\end{equation}
Then, due to the additional $\mu - t$ in the denominator of the error term~\eqref{eq:err2} compared to the error term of~\eqref{eq:se_split}, the bound for \eqref{eqn:bnd_roots_differ} has an additional factor of $(\lambda_k - \mu)^{-1}$ in the constant $C_k$.
\end{proof}

To conclude the above discussion on the two approximations of the secular equation, we present Figure~\ref{fig:se_and2truncatedSE}. In this figure, we construct a matrix of size $n=4$ with eigenvalues $0.1,0.2,0.3,0.4$ and random orthogonal eigenvectors. To form the truncated equations, we use $m=2$ and $\mu = \mu_{mean}$ of \eqref{eq:mu_mean}, which in this case satisfies $\mu=0.15$. The figure depicts the two truncated secular equations $w_1$ of \eqref{eq:TSE} and $w_2$ of \eqref{eq:CTSE}, for a rank-one update with $\rho>0$, alongside with the original secular equation of \eqref{eqn:secular_equation}. The two roots that are approximated are on the white part of the figure. We zoom in on a neighbourhood of the second root of the secular equation $t_2$, to observe how the second order approximation has a closer root than the root of the first order approximation, as theory suggests. The other two roots (that are not approximated) are on the grey part of the figure, where the asymptotic behaviour around $\mu$ of the two approximations is demonstrated.
\begin{figure}  
    \centering
        \includegraphics[width=.55\textwidth]{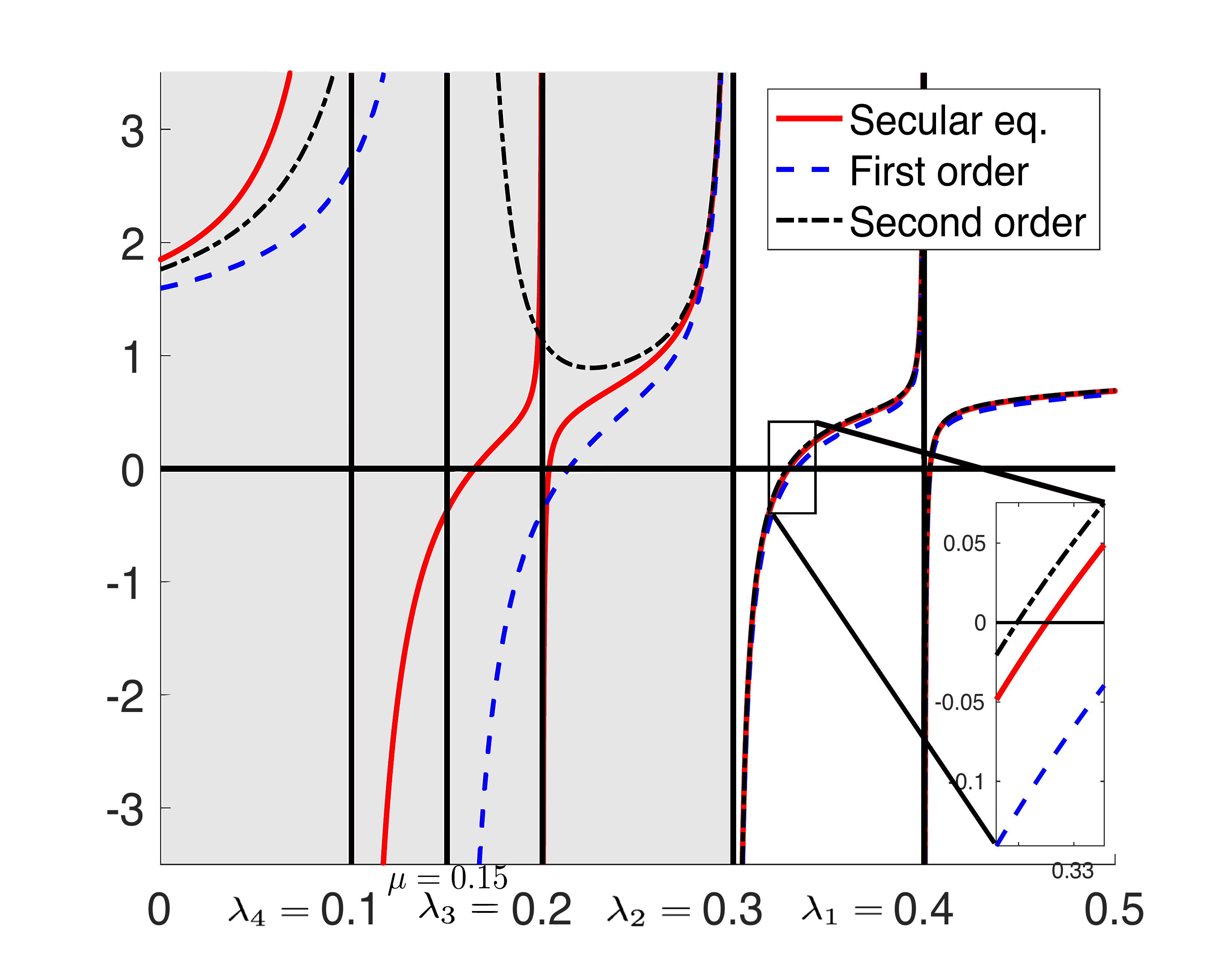}
        \caption{The secular equation \eqref{eqn:secular_equation} and its two approximations: the first order $w_1$ of \eqref{eq:TSE} and the second order $w_2$ of \eqref{eq:CTSE}. The original matrix has four eigenvalues at $0.1,0.2,0.3$ and $0.4$ and a rank-one update with $\rho>0$. The approximations use $m=2$, and $\mu = \mu_{mean}=0.15$. In the lower right corner, we zoom-in to a small neighborhood of the second root.} \label{fig:se_and2truncatedSE}
\end{figure}

We address once again the choice of $\mu$. Setting $\mu = 0$ implies that the eigenvalues $\lambda_j$, $j=m+1,\ldots,n$ are assumed to be small. Then, we get according to Proposition~\ref{prop:CTSE_roots} an improved error of $O(\lambda_{m+1}^2)$. Nevertheless, in this case of a second order approximation to the secular equation \eqref{eq:CTSE}, an improved method for choosing $\mu$ is possible by minimizing an upper bound on~\eqref{eq:err2}. For simplicity, we further assume that $\mu \leq \lambda_{m+1}$. Since the denominator in \eqref{eq:err2} depends on $t$ (specifically, on the yet to be calculated approximated eigenvalues), we bound it by some constant that will apply to all eigenvalues simultaneously. One such bound is $e(t ; \mu) \leq \rho (\lambda_{m} - \lambda_{m+1})^{-3} \sum_{i=m+1}^{n} z_i^2 (\lambda_i - \mu)^2$ which holds for all $t \in (\lambda_{m}, \lambda_{1} + \rho)$. We would then like to minimize $\sum_{i=m+1}^{n}z_i^2(\lambda_i - \mu)^2$. By standard methods we get the minimizer
\begin{equation} \label{eq:mu_opt}
\mu_\ast = \frac{\sum_{i=m+1}^nz_i^2\lambda_i}{ \sum_{i=m+1}^{n}z_i^2} = \frac{s}{1 - \sum_{i=1}^{m}z_i^2} ,
\end{equation}
where $s$ is defined in~\eqref{eq:s}. The minimizer $\mu_\ast$ is essentially a weighted mean of the unknown eigenvalues (and thus obeys the assumption $\mu \leq \lambda_{m+1}$). Unlike $\mu_{mean}$, this variant does not require the knowledge of $\operatorname{tr}(A)$ but rather a few matrix-vector evaluations to calculate $s$ of~\eqref{eq:s}. Interestingly, note that when using $\mu = \mu_\ast$ we have 
\begin{equation}
w_2(t ; \mu_\ast) = w_1(t ; \mu_\ast) ,
\end{equation}
meaning that we have a second order approximation in both formulas. 

Next, we address the problem of eigenvectors estimation.

\subsection{Truncated formulas for the eigenvectors}
In this section, we introduce two approximations to the eigenvectors formula~\eqref{eqn:EigenvaectorFormula}. These are analogous to the approximations to the secular equation from the previous section. The two approximations are designed to use only the $m$ leading eigenvalues and their eigenvectors, and differ in accuracy and time complexity. 
 
A naive way to truncate the eigenvectors formula \eqref{eqn:EigenvaectorFormula} is by calculating
\begin{equation} \label{eq:naive_TEF}
\widetilde{p_i} = Q^{(m)}(\Delta_i^{(m)})^{-1}(Q^{(m)})^Tv , \quad  \Delta_i^{(m)} = \diag\left(\lambda_1 - t_i,\ldots,\lambda_m - t_i\right), \quad  i=1,\ldots,m ,
\end{equation}
followed by a normalization, where $t_i$ are the roots of the secular equation (the updated eigenvalues) in descending order. We now ignore for a while the normalization and focus on the unnormalized vectors, namely (see \eqref{eqn:EigenvaectorFormula}),
\begin{equation} \label{eqn:simplify_eigenvector_formula}
\begin{split}
{p_i} & = Q(\Lambda - t_i I)^{-1}Q^Tv \\
  & = \begin{bmatrix}
        q_1 & ... & q_n
     \end{bmatrix}
    \begin{bmatrix}
        \frac{\langle q_1, v \rangle}{\lambda_1 - t_i} \\
        ... \\
        \frac{\langle q_n, v \rangle}{\lambda_n - t_i} \\
    \end{bmatrix} = \sum_{k=1}^{n} \frac{\langle q_k, v \rangle}{\lambda_k - t_i}q_k  \\
    & = \underbrace{ \sum_{k=1}^{m} \frac{\langle q_k, v \rangle}{\lambda_k - t_i}q_k}_{\text{known}} + \underbrace{ \sum_{k=m+1}^{n} \frac{\langle q_k, v \rangle}{\lambda_k - t_i}q_k}_{\text{unknown}} .
\end{split}
\end{equation}

Note that the sum of unknown terms, without weights, is accessible as
\begin{equation} \label{eqn:sum_unknown_proj}
\sum_{k=m+1}^{n} \langle q_k, v \rangle q_k = \sum_{k=1}^{n} \langle q_k, v \rangle q_k - \sum_{k=1}^{m} \langle q_k, v \rangle q_k  = v - \sum_{k=1}^{m}q_kq^T_kv = v - Q^{(m)} (Q^{(m)})^Tv .
\end{equation}

Again, we denote by $\mu$ a fixed parameter whose purpose is to approximate the unknown eigenvalues. Having the $m$ leading eigenvectors in $Q^{(m)}$, and recalling  that $\Delta_i^{(m)} = \diag\left(\lambda_1-t_i,\ldots,\lambda_n-t_i \right)$, we define the first order truncated eigenvectors formula for $1\le i \le m$ as
\begin{equation} \label{eqn:TEF}
\widetilde{p_i} = Q^{(m)}(\Delta_i^{(m)})^{-1}(Q^{(m)})^Tv + \frac{1}{\mu - t_i}r , \quad  r = v - Q^{(m)} (Q^{(m)})^Tv .
\end{equation}
The second order truncated eigenvectors formula, which is the eigenvectors analogue of \eqref{eq:CTSE}, is given by
\begin{equation} \label{eq:CTEF}
\widetilde{p_i} = Q^{(m)}(\Delta_i^{(m)})^{-1}(Q^{(m)})^Tv + \left(\frac{1}{\mu - t_i} + \frac{\mu}{(\mu - t_i)^2}\right)r - \frac{1}{(\mu - t_i)^2}Ar , \quad   r = v - Q^{(m)} (Q^{(m)})^Tv .
\end{equation}
Note that $r$ and $Ar$ are constant vectors and can be computed once for all $1\le i \le m$.

The error bounds of both formulas \eqref{eqn:TEF} and \eqref{eq:CTEF} are summarized in the following theorem.
\begin{theorem} \label{thm:err_bnd_trunc_vecs}
Let $A = Q \Lambda Q^T $ be an ${n \times n}$ real symmetric matrix with $m$ known leading eigenvalues $\lambda_1 > \cdots > \lambda_m$ and known corresponding eigenvectors $q_1, \ldots, q_m$. The $m$ leading eigenvectors $p_1,\ldots,p_m$ of the rank-one update $A+\rho vv^T$, with $\norm{v} = 1$ and $\rho \in \mathbb{R}$ can be approximated by \eqref{eqn:TEF} or \eqref{eq:CTEF}, given their associated leading eigenvalues $t_1,\ldots,t_m$ and a fixed scalar $\mu$, such that the approximations satisfy:
\begin{enumerate}
\item For $\widetilde{p_i}$ of \eqref{eqn:TEF}, 
\begin{equation} \label{eqn:EVfirstbound}
 \norm{p_i - \widetilde{p_i}} \le C_i \max_{m+1 \leq i \leq n} | \lambda_{i} - \mu| , \quad C_i \le \abs{\mu -t_i}^{-1}\abs{\lambda_{m+1}-t_i}^{-1} . 
\end{equation}
\item For $\widetilde{p_i}$ of \eqref{eq:CTEF}, 
\begin{equation} \label{eqn:EVsecondbound}
 \norm{p_i - \widetilde{p_i}} \le C_i \max_{m+1 \leq i \leq n}|\lambda_{i} - \mu|^2 , \quad C_i \le \abs{\mu -t_i}^{-2}\abs{\lambda_{m+1}-t_i}^{-1} . 
 \end{equation}
\end{enumerate}
\end{theorem}
\begin{proof}
For~\eqref{eqn:EVfirstbound}, by \eqref{eqn:simplify_eigenvector_formula} and \eqref{eqn:TEF} we have 
\begin{equation*}
e_i = \norm{p_i- \widetilde{p_i}} = \norm{\sum_{k=m+1}^{n} \frac{\langle q_k, v \rangle}{\lambda_k - t_i}q_k + \frac{1}{\mu - t_i}r }.
\end{equation*}
Similarly to \eqref{eqn:summand_first_app},
  $\frac{1}{\lambda_k - t_i} = \frac{1}{\mu-t_i} +  \frac{(\mu - \lambda_k)}{(\mu-t_i)(\lambda_k - t_i)}$, and using \eqref{eqn:sum_unknown_proj} and the orthogonality of $q_i$, we have
\begin{align}   \label{eq:err_vec1}
e_i^2  & =   \norm{\sum_{k=m+1}^{n} \frac{(\mu - \lambda_k) \langle q_k, v \rangle}{(\mu - t_i)(\lambda_k - t_i)}q_k  }^2 \nonumber  = \sum_{k=m+1}^{n} \left( \frac{(\mu - \lambda_k) \langle q_k, v \rangle}{(\mu - t_i)(\lambda_k - t_i)} \right)^2\nonumber  \\
& \leq \frac{1}{(\mu - t_i)^2(\lambda_{m+1} - t_i)^2} \sum_{k=m+1}^{n}  (\mu - \lambda_k)^2 \langle q_k, v \rangle ^2 \nonumber  \\
& = \frac{1}{(\mu - t_i)^2(\lambda_{m+1} - t_i)^2} \sum_{k=m+1}^{n}  (\mu - \lambda_k)^2 z_k^2 .    
\end{align}   
Since $\|Q^Tv\| = 1$, taking the square root of~\eqref{eq:err_vec1} gives us the bound.

For~\eqref{eqn:EVsecondbound}, we use~\eqref{eqn:ImprovedExpansion}, and recall that $Aq_k = \lambda_kq_k$ which in this case means
\begin{equation*}
\frac{\mu}{(\mu - t_i)^2}r - \frac{1}{(\mu - t_i)^2}Ar = \frac{1}{(\mu - t_i)^2} \sum_{k=m+1}^{n} \langle q_k, v \rangle (\mu - \lambda_k) q_k    .
\end{equation*}
Thus, we have for the corrected formula \eqref{eq:CTEF} with exact eigenvalues,
\begin{equation}
\norm{p_i- \widetilde{p_i}}^2  = \norm{\sum_{k=m+1}^{n} \frac{(\mu - \lambda_k)^2 \langle q_k, v \rangle}{(\mu - t_i)^2(\lambda_k - t_i)}q_k  }^2 = \sum_{k=m+1}^{n} \left( \frac{(\mu - \lambda_k)^2 \langle q_k, v \rangle}{(\mu - t_i)^2(\lambda_k - t_i)} \right)^2   , \end{equation}
and the second claim follows as before.
\end{proof}

As with the eigenvalues, under the low rank assumption ($\mu = 0$), Theorem~\ref{thm:err_bnd_trunc_vecs} guarantees errors of $O(|\lambda_{m+1}|)$ and $O(\lambda_{m+1}^2)$ for \eqref{eqn:TEF} and \eqref{eq:CTEF}, respectively. The value of $\mu$ which minimizes the bound on the last term in \eqref{eq:err_vec1} is $\mu_\ast$ of \eqref{eq:mu_opt}, and is thus expected to provide a better approximation than $\mu = 0$. Experimental results have shown that the choice $\mu = \mu_{mean}$ of \eqref{eq:mu_mean} is competitive with $\mu_\ast$ while being slightly faster to compute. Note that the approximate formulas \eqref{eqn:TEF} and \eqref{eq:CTEF} will generally not produce an orthogonal set of vectors. In that case, a re-orthogonalization procedure may be used. This issue is discussed in Appendix~\ref{app:loss_orth}.

\subsection{Algorithm summary}

Given a parameter $\mu$, we have provided first and second order truncated approximations to the secular equation and corresponding formulas for the eigenvectors. As for $\mu$, we suggested three choices. If the matrix is low-rank, choose $\mu = 0$. Otherwise, choose either $\mu_\ast$ which minimizes the error term, or $\mu_{mean}$ which is faster to compute. We summarize the previous subsections in Algorithm~\ref{alg:trunc}, which computes the symmetric rank-one update with partial spectrum.

\begin{algorithm}[ht]
\caption{rank-one update with partial spectrum}
\label{alg:trunc}
\begin{algorithmic}[1]
\REQUIRE $m$ leading eigenpairs $\{(\lambda_i,q_i)\}_{i=1}^m$ of a symmetric matrix $A$, a vector $v \in \mathbb{R}^n$ with $\|v\| = 1$ and a scalar $\rho > 0$
\ENSURE An approximation $ \{(\widetilde{t_i},\widetilde{p_i})\}_{i=1}^m$ of the eigenpairs of $A + \rho vv^T$
\STATE Choose a parameter $\mu$ (i.e., $\mu=0$, \eqref{eq:mu_mean} or \eqref{eq:mu_opt}). \label{alg1:line1}
\STATE Calculate the $m$ largest roots  $\{(\widetilde{t_i}\}_{i=1}^m$ of a truncated secular equation (either \eqref{eq:TSE} or \eqref{eq:CTSE}) \label{alg1:line2}
\FORALL { $\{{q_i}\}_{i=1}^m$ }
\STATE  find $\widetilde{p_i}$ by a truncated eigenvectors formula (either \eqref{eqn:TEF} or \eqref{eq:CTEF}) \label{alg1:line4}
\ENDFOR
\end{algorithmic}
\end{algorithm}
A complexity analysis of Algorithm~\ref{alg:trunc} is provided in Appendix~\ref{sec:analysis_1}. 

\section{Updating the graph Laplacian for out-of-sample extension} \label{sec:updating_problem} 

In this section, we introduce an application of the rank-one update scheme of Section~\ref{sec:rank_one_update}, to the problem of out-of-sample extension of the graph Laplacian. We start by formulating the problem, and then justify the use of a rank-one update by proving that a single point extension of the graph Laplacian is close to a rank-one perturbation. We conclude the section with a few algorithms, which are demonstrated numerically in Section~\ref{sec:numeric}.

\subsection{Preliminaries and problem formulation} \label{sec:prel}
We begin by introducing the notation and the model for the extension problem. Given a set of discrete points $\mathcal{X} = \{ x_i \}_{i=1}^n \subset \mathbb{R}^d $, we define a weighted graph whose vertices are the given points. An edge is added to the graph if its two vertices are ``similar". The common ways of defining ``similar" include:
\begin{enumerate}
	\item $k$-nearest neighbours (kNN) --  Vertex $i$ and $j$ are connected iff $i$ is within the kNN of $j$ or vice versa. 
	\item $\delta$-neighbourhood --  Vertex $i$ and $j$ are connected iff $\|x_i - x_j\| < \delta$ for some $\delta > 0$. 
\end{enumerate}
Each edge in the graph is assigned a weight, usually determined by a kernel function. A kernel function $K$ is a symmetric function $K \colon \mathbb{R}^d \times  \mathbb{R}^d \to \mathbb{R} $. The weight on the edge between vertices $i$ and $j$ is set to $w_{ij} = K(x_i, x_j)$. A kernel is said to be radial if 
\begin{equation} \label{eqn:radialKer}
K(x,y) = g \left(\norm{x-y} \right), \quad x,y \in \mathbb{R}^d ,
\end{equation}
for a non-negative real function $g$. One common choice of a kernel is the heat kernel (also known as the Gaussian kernel) that induces the weights
\begin{equation} \label{eqn:guass_ker}
w_{ij} = \exp \big( -\frac{\|x_i - x_j\|^2}{\varepsilon}  \big),
\end{equation}
for some fixed width parameter $\varepsilon > 0$.

Given the weight matrix $W = \{w_{ij} \}$ and its corresponding (diagonal) degrees matrix $D$ whose diagonal is $D_{ii} = \sum_{j=1}^{n}W_{ij}$, the graph Laplacian is typically defined as either $L=D^{-1}W$ (random walk graph Laplacian) or $L = D^{-\frac12}WD^{-\frac12}$ (Symmetric normalized graph Laplacian) \cite{belkin2003laplacian}. Note that most authors define the symmetric normalized graph Laplacian as $L = I - D^{-\frac12}WD^{-\frac12}$. The latter definition of the graph Laplacian merely applies an affine transformation to the eigenvalues of the graph Laplacian and does not change the corresponding eigenvectors. Since we define our method to act on the largest eigenvalues, we prefer using $L = D^{-\frac12}WD^{-\frac12}$. Recall that $W$ and $L$ are $n \times n$ matrices where $n$ is the number of samples in $\mathcal{X}$. In the following, we consider only the case of the symmetric normalized graph Laplacian. Nevertheless, similar results can be obtained for the random walk graph Laplacian, as it satisfies a similarity relation with the symmetric graph Laplacian. Henceforth, unless otherwise stated, by referring to the ``graph Laplacian" we mean the symmetric normalized graph Laplacian.

\noindent
We now formulate the out-of-sample extension of the graph Laplacian. Let $\mathcal{X} = \{ x_i \}_{i=1}^n \subset \mathbb{R}^d $ and let $x_0 \not\in \mathcal{X}$ be a new point in $\mathbb{R}^d $. Denote by $L_0$ the graph Laplacian constructed from $\mathcal{X}$ using a given kernel. Assume the top $m$ eigenvalues and eigenvectors of $L_0$ are known ($m<n$). The out-of-sample extension problem is to find the top $m$ eigenpairs of $L_1$, the graph Laplacian constructed from $\mathcal{X} \cup \{x_0\}$.

The out-of-sample extension problem is reduced to a symmetric rank-one update as follows. With a slight abuse of notation, we also denote by $L_0$ the original graph Laplacian to which we added $x_0$ as an isolated vertex. That is, $L_0$ is now an $(n+1) \times (n+1)$ matrix whose first row and column correspond to the point $x_0$: the first row and column have $1$ on the diagonal and $0$ otherwise. Note that the dimensions of this augmented $L_0$ are identical to the dimensions of $L_1$. We will argue that the difference matrix $\Delta L = L_1 - L_0$ is very close to being rank-one (a claim that will be formulated and proven in the next section). In other words, by looking at  $\rho = \lambda_1(\Delta L)$ (the leading eigenvalue of $\Delta L$), and its associated eigenvector $v = q_1(\Delta L)$, we estimate the leading eigenpairs of $L_1$ using the proxy
\begin{equation} \label{eq:proxy}
\widetilde{L}_1 = L_0 + \rho vv^T .
\end{equation}

An illustration of the out-of-sample extension problem is given in Figure~\ref{fig:newnode}.
\begin{figure} 
    \centering
    \begin{subfigure}[b]{0.4\textwidth}
        \includegraphics[width=\textwidth]{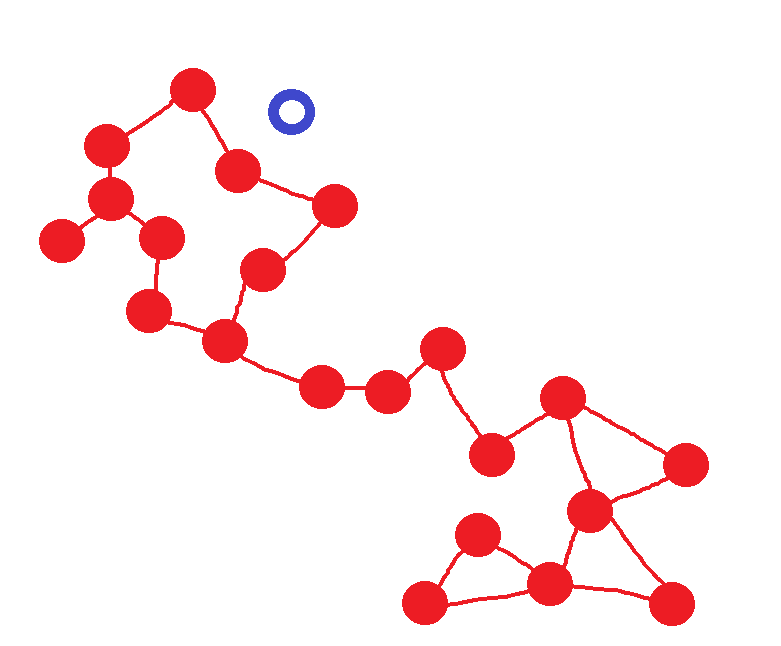}
        \caption{The graph corresponding to $L_0$, with $x_0$ (empty dot) disconnected.}
    \end{subfigure}
    \begin{subfigure}[b]{0.4\textwidth}
        \includegraphics[width=\textwidth]{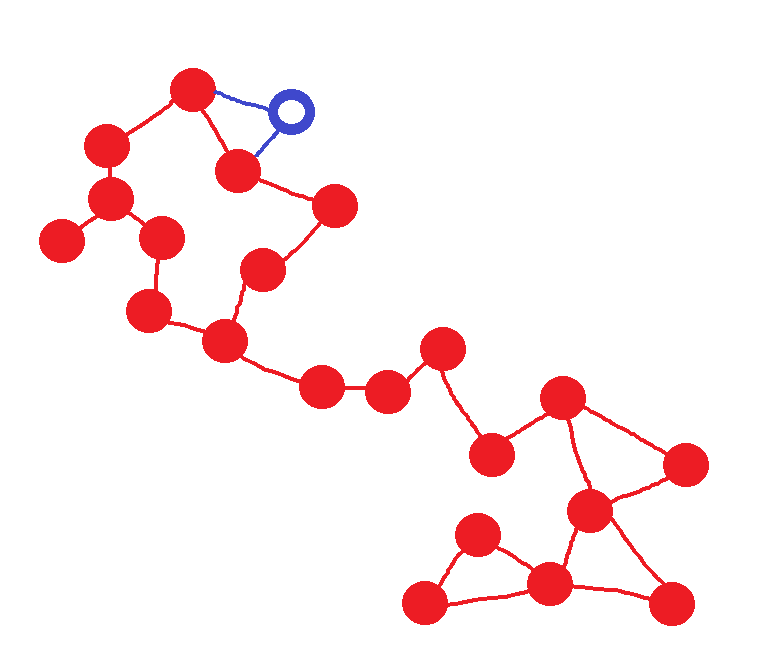}
        \caption{The graph corresponding to $L_1$, with $x_0$ (empty dot) connected.}      
    \end{subfigure}  
    \caption{Adding a new sample point $x_0$ to the graph. }\label{fig:newnode}
\end{figure}

\subsection{Updating the graph Laplacian is almost a rank-one perturbation}

As described in Section~\ref{sec:prel}, the weights on the edges of the graph are determined by a kernel. For our subsequent claims, we will require that our kernel is radial with $g(0) > 0$, and that in some neighbourhood of $0$ its derivative is bounded, that is $\abs{\frac{d}{dx}g} < M$ for some $M > 0$. These requirements are not too restrictive, as they are met by most common kernels used, such as the heat kernel.  

In the following analysis, we consider graphs constructed using $\delta$-neighbourhoods (see Section~\ref{sec:prel}). As we will see next, the analogue for kNN is straightforward. For $\delta$-neighbourhoods, we require the parameter $\delta$ to be ``small enough", and more specifically, to satisfy $\delta < \frac{g(0)}{2M}$. This assumption is not too restrictive as the purpose of constructing similarity graphs is to model the local
neighbourhood relationships between the data points~\cite{von2007tutorial}.

We denote by $k$ the minimal number of neighbours of a vertex. In addition, we denote by $c_1 \ge 1$ a constant such that $c_1 \cdot k$ is the maximal number of neighbours of a vertex (we assume that $c_1$ is independent of $k$). Denote by $\sigma_i(X)$, $i=1,\ldots,n$ the singular values of a squared symmetric matrix $X$ (in descending order). We now present the main theoretical result of this section.
\begin{theorem} \label{thm:rankone}
Under the assumptions and notation described above, let $L_0$ and $L_1$ be two graph Laplacians before and after the addition of a new vertex, respectively. Then, there exists a constant $\beta$, independent of $k$, such that 
\[  \sigma_1(L_1 - L_0) = 1 - \frac{\beta}{k} \qquad \text{ and } \qquad \sigma_i(L_1 - L_0) =   \frac{\beta}{k}, \quad i \geq 2 . \] 
\end{theorem}
Theorem~\ref{thm:rankone} shows that for large enough $k$, $\sigma_1(L_1 - L_0) \approx 1$ and $\sigma_i(L_1 - L_0) \approx 0$, $ i \geq 2$. In other words, $\Delta L = L_1 - L_0$ is indeed close to being rank-one.

\subsection{Proof of Theorem~\ref{thm:rankone}}

The proof is divided into a few steps. First, we adapt a classical result from perturbation theory called Weyl's Theorem \cite{stewart1998perturbation} to our setting for an initial bound on the singular values of $\Delta L$. Then, we use our assumptions to derive, based on the specific structure of the graph Laplacian, the required constants and bounds to use in the main body of the proof.

From classical perturbation theory we have the following result regarding the singular values of a matrix.
\begin{theorem}[Weyl's Theorem] \label{thm:weyl}
Let $S,E \in \mathbb{R}^{n \times n}$. Then, for all $1 \leq i \leq n$ we have
\[ \abs{\sigma_i(S+E) - \sigma_i(S)} \leq \norm{E}_2. \]
\end{theorem}
As it turns out, for the special case where $S$ is diagonal, we can further improve the above estimation.
\begin{theorem} \label{thm:extendweyl}
Let $S\in \mathbb{R}^{n \times  n}$ be a diagonal matrix whose diagonal entries are different from each other, and let $E \in \mathbb{R}^{n \times 
n}$. Assume, without loss of generality, that the diagonal entries of $S$ are given in a descending order of magnitude. Denote $E = \left( e_{ij}\right)$. Let $\eta > 0$ be such that $\|E\|_2 < \eta$. Then, for a small enough $\eta$, there exists $c_H = c_H(S, \eta) > 0$ independent of $E$ so that for all $1 \leq i \leq n$,
\[ \abs{\sigma_i(S+E) - \sigma_i(S) -  \operatorname{sign}(S_{ii})e_{ii}} \leq  c_H \norm{E}^2_F . \]
\end{theorem}	
The proof of Theorem~\ref{thm:extendweyl} is given in Appendix~\ref{sec:ProofOfWeyl2thm}.

Recall that our aim is to bound the singular values of $\Delta L = L_1 - L_0$, that is of the difference matrix between the graph Laplacians before and after the insertion of a new vertex. To apply Theorem \ref{thm:extendweyl}, we denote $S = \operatorname{diag}(\Delta L)$ and $E = \Delta L - S$. Assume we permuted the indices of the vertices of the graph such that the diagonal entries of $S$ are in descending order. Note that in our specific case, the diagonal entries of $E$ are in fact zero. Therefore, by Theorem~\ref{thm:extendweyl} there exists $c^\prime>0$ so that
\begin{equation} \label{eqn:dls}
\abs{\sigma_i(\Delta L) - \sigma_i(S)} \leq  c^\prime \norm{E}^2_F.
\end{equation}
It is clear now that estimating $\norm{E}_F$ will provide us the relation between the singular values of $\Delta L$ and $S$. 

We start by examining $\Delta L$; its only nonzero elements are the ones affected by the introduction of the new vertex. There are at most $c_1k$ such rows, each consists of at most $c_1k$ nonzero elements by assumption. Thus, the total number of elements changed in these rows is at most $c_1k \times c_1k = c_1^2k^2$. Due to symmetry, the same goes for the columns, thus, we have at most $c_1^2k^2 + c_1^2k^2 = 2c_1^2k^2$ changed entries. In other words, using the convention that $\operatorname{nnz}(X)$ is the number of nonzero elements of a matrix $X$, we have that 
\begin{equation} \label{eqn:boundnnz}
  \operatorname{nnz}(\Delta L) \leq (2c_1^2)k^2  . 
\end{equation}
An element-wise estimation of the entries of the graph Laplacian, stating they are of order $\frac{1}{k}$, is given next.
\begin{lemma} \label{lemma:bounding_elementwise}  
Let $L =(\ell_{i,j})$ be a graph Laplacian, calculated using $\delta$-neighbourhoods, using a radial kernel $g$ with a bounded derivative $\abs{\frac{d}{dx}g} < M$ such that $g(0) > 2M\delta$. Then, 
\begin{equation}
 \frac{1}{c_1c} \cdot \frac{1}{k} < \ell_{ij} < \frac{c}{k}  , \quad 1\le i,j \le n , \quad c = 1+ \frac{g(0)}{M\delta}. 
\end{equation}
\end{lemma}
\begin{proof}
Let $\alpha_{ij} = \norm{x_i - x_j}$. Then, using Lagrange's remainder theorem, for each entry $w_{ij}$ of the weight matrix $W$ there exists $\xi_{ij}$ such that
\begin{equation} \label{eq:wij}
w_{ij} = g(\|x_i - x_j\|) =  g(0) + \frac{d}{dx}g(\xi_{ij})\alpha_{ij} . 
\end{equation}
Since $\alpha_{ij} < \delta$ and $\abs{\frac{d}{dx}g} < M$, we have that an upper bound on $w_{ij}$ is $g(0) + M\delta$. On the other hand, $g(0) > 2M\delta$, so we get the bounds
\begin{equation} \label{eq:w_bounds}
M\delta < w_{ij} < g(0) + M\delta.
\end{equation}
The $ij$-th entry of the graph Laplacian is
\begin{equation} \label{eq:gl_entry}
\ell_{ij} = \frac{w_{ij}}{\sqrt{\sum_p w_{pj}} \cdot \sqrt{\sum_p w_{ip}}} , 
\end{equation}
where the two sums are taken over all the neighbours of the $i$-th and $j$-th vertices. 
\noindent
The number of neighbours of each vertex is at least $k$ so 
\begin{equation} \label{eq:sum_gl}
\sum_p w_{pj} \ge k M \delta  .
\end{equation}
Therefore,
\begin{equation}
\ell_{ij} < \frac{g(0) + M\delta}{\sqrt{kM\delta}\sqrt{kM\delta}} = \frac{g(0) + M\delta}{M\delta k} .
\end{equation}
Similarly, using the upper bound  $c_1 k$ on the number of neighbours we get
\begin{equation}
\ell_{ij} > \frac{M\delta}{c_1(g(0) + M\delta)k} .
\end{equation}
\end{proof}
An immediate conclusion from Lemma~\ref{lemma:bounding_elementwise} is the following.
\begin{lemma} \label{lemma:entries}
The entries of $\Delta L$ are of order $O(\frac1k)$ except for the first entry $(\Delta L )_{11}$ which is $-1 + O(\frac1k)$.
\end{lemma}
\begin{proof}
Denote by $l^0_{ij}$ and $l^1_{ij}$ the $(i,j)$ entry of $L_0$ and $L_1$, respectively $(1 \leq i,j \leq n + 1)$. By Lemma~\ref{lemma:bounding_elementwise}, for $(i,j) \neq (1,1)$ both $l^0_{ij}$ and $ l^1_{ij}$ are $O(\frac1k)$, and thus the entries of $\Delta L$, which are of the form  $l^1_{ij} - l^0_{ij}$, are $O(\frac1k)$. In the case $i=j=1$, by construction $l^0_{11} = 1$ and thus $(\Delta L )_{11} = l^1_{11} - l^0_{11} = -1 + O(\frac1k)$. 
\end{proof}
It follows that $\Delta L$ is dominated by its first entry, and so it is somewhat unsurprising that it is close to being rank-one. A sharper element-wise bound is given in the following lemma.
\begin{lemma} \label{lem:boundsize}
The entries of $\Delta L$ that are not on the first row/column are smaller in magnitude than $\frac{c^2}{2k^2}$, $c = 1+ \frac{g(0)}{M\delta}$.
\end{lemma}
The proof of Lemma~\ref{lem:boundsize} is given in Appendix~\ref{sec:ProofOfLemmaBoundSize}.

According to~\eqref{eqn:boundnnz}, $\Delta L$ has at most $2c_1^2k^2$ nonzero elements. At most $2c_1k$ of those are on the first row and column. The magnitude of these elements is at most $\frac{c}{k}$. The rest of the nonzero elements, in light of Lemma~\ref{lem:boundsize}, have magnitude of at most $\frac{c^2}{2k^2}$.

Consider the non-zero elements of $E = \Delta L - \operatorname{diag}(\Delta L)$. The ones that are on the first row/column have magnitude of at most $\frac{c}{k}$, and there are at most $c_1k$ of them. Within the elements that are not on the first row/column, based on \eqref{eqn:boundnnz}, there are at most $2c_1^2k^2$ nonzero elements, and by Lemma~\ref{lem:boundsize}, their magnitude is at most $\frac{c^2}{2k^2}$ in size. Therefore, we can bound the Frobenius norm of $E$ as
\begin{equation}
\|E\|_F^2 \leq 2c_1k \cdot \Bigg( \frac{c}{k} \Bigg) ^2  +   2c_1^2k^2 \Bigg( \frac{c^2}{2k^2} \Bigg) ^2 = \frac{2c^2c_1}{k} + \frac{c^4c_1^2}{2k^2} < \frac{2c^4c_1^2}{k} + \frac{c^4c_1^2}{k} = \frac{3c^4c_1^2}{k} . 
\end{equation}
Namely,
\begin{equation}
 \|E\|_F \leq \sqrt{3}c^2c_1\frac{1}{\sqrt{k}} .
\end{equation}

We finally prove our main theorem.
\begin{proof}[Proof of Theorem~\ref{thm:rankone}]
Recall that $S = \operatorname{diag}(\Delta L)$. Denoting $ \tilde{c} = 3c^4c_1^2c'$, by \eqref{eqn:dls}, 
\begin{equation} \label{eq:bound}
|\sigma_i(\Delta L) - \sigma_i(S)| < c' \|E\|_F^2 < c' \cdot \Bigg( \frac{\sqrt{3}c^2c_1}{\sqrt{k}}  \Bigg)^2 = \frac{3c^4c_1^2c'}{k} =   \frac{\widetilde{c}}{k} , \quad 1 \leq i \leq n .
\end{equation}
The largest singular value of $S$ is the absolute value of its largest entry, and by Lemma~\ref{lemma:entries} we have $\sigma_1 (S) = |(\Delta L)_{11}| $ for large enough $k$. By Lemma~\ref{lemma:bounding_elementwise} and~\eqref{eq:bound},
\begin{equation}
\sigma_1 (\Delta L) < \sigma_1 (S) +  \frac{\widetilde{c}}{k} < 1 - \frac{1}{c_1ck} + \frac{\widetilde{c}}{k} = 1 - \frac{1 + \widetilde{c}c_1c}{c_1ck} ,
\end{equation}
and
\begin{equation}
\sigma_1 (\Delta L) > \sigma_1 (S) - \frac{\widetilde{c}}{k} > 1 - \frac{c}{k} - \frac{\widetilde{c}}{k} = 1 - \frac{c + \widetilde{c}}{k} .
\end{equation}
Namely,  $\sigma_1(\Delta L)$ is of order $1 - \frac{1}{k}$. The other singular values of $S$ are the other diagonal entries, which are at most $\frac{c}{k}$, by Lemma~\ref{lemma:bounding_elementwise}. Thus, by~\eqref{eq:bound} we have\begin{equation}
\sigma_i (\Delta L) < \sigma_i (S) +  \frac{\widetilde{c}}{k} < \frac{c}{k} + \frac{\widetilde{c}}{k} =  \frac{ \widetilde{c} + c}{k} , \quad i \geq 2 ,
\end{equation}
which shows that $\sigma_i (\Delta L)$ is of order $\frac{1}{k}$ as required.
\end{proof}

\subsection{Rank-one update and error analysis} \label{sec:alg_gl}

We next discuss the required adjustments for applying Algorithm~\ref{alg:trunc} of rank-one update to the out-of-sample extension problem. We wish to find the best rank-one approximation of $\Delta L = L_1 - L_0$, which we denote by $(\Delta L)_1$. Such an approximation requires recovering the largest singular value of $\Delta L$, denoted by $\sigma_1(\Delta L)$, and its corresponding left and right singular vectors, denoted by $v_L$ and $v_R$ respectively. 

Denote by $(\rho, v)$ the top eigenpair of $\Delta L$. Since $\Delta L \neq 0$ is symmetric, if its largest eigenvalue $\lambda$ is positive, then $\sigma_1(\Delta L) = \lambda$ and $v_L = v_R = v$. If $\lambda < 0$ then $\lambda = -\sigma_1(\Delta L)$ and $v_L = -v_R = v$. Thus, in both cases, the best rank-one approximation of $\Delta L$ is
\begin{equation} \label{eqn:deltaL1}
(\Delta L)_1 = \sigma_1(\Delta L)v_Lv_R^T = \rho vv^T.
\end{equation} 
The out-of-sample extension algorithm, together with a perturbation correction that will be introduced shortly, is described in Algorithm~\ref{alg:extension_algo_corr}. The approximated eigenpairs returned by the algorithm are affected by two types of error: the error induced by truncating the rank-one update equations, which was discussed in Section~\ref{sec:rank_one_update}, and the error induced by the rank-one approximation of $\Delta L$, which we examine now. 

To analyze the error of Algorithm~\ref{alg:extension_algo_corr} and to further improve our approximation for the updated eigenvalues and eigenvectors, we use two classical results from matrix perturbation theory. These results, Lemma~\ref{lemma:pert1} and Lemma~\ref{lemma:pert2}, are given without proofs, and the interested reader is referred to \cite[Chapter~4]{byron2012mathematics}. As before, we denote by $q_i(X)$ a normalized eigenvector that is associated with the $i$-th largest eigenvalue of $X$.
\begin{lemma} \label{lemma:pert1}
Let $A, B \in \mathbb{R}^{n \times n}$ be symmetric matrices. The following holds for all $1 \le i \le n$,
\begin{enumerate}
	\item $|\lambda_i(A +  B) - \lambda_i(A)| = O( \norm{B} )$.
	\item $\norm{q_i(A + B) - q_i(A)} = O( \norm{B} )$.
\end{enumerate}
\end{lemma}

\noindent
Using Lemma~\ref{lemma:pert1}, let 
\begin{equation} \label{eqn:coreection_explained}
A + B = L_1 = L_0 + \Delta L, \quad A = L_0 + (\Delta L)_1,
\end{equation}
where $(\Delta L)_1$ is defined in~\eqref{eqn:deltaL1}. Then, the rank-one update~\eqref{eq:proxy} induces an error of order  $\norm{\Delta L - (\Delta L)_1} =\sigma_2(\Delta L)$, and by Theorem~\ref{thm:rankone} we conclude that this error is of order
\[ \sigma_2(\Delta L) = O\left(\frac{1}{k}\right) . \]
Similarly to Section~\ref{sec:rank_one_update}, we can obtain higher order approximation using a further correction, based on the following result.
\begin{lemma} \label{lemma:pert2}
Let $A, B \in \mathbb{R}^{n \times n}$ be symmetric matrices. The following holds for all $1 \le i \le n$,
\begin{enumerate}	
	\item $\Big| \lambda_i(A + B) - \big[ \lambda_i(A) + q_i^T(A)Bq_i(A) \big]\Big| =  O(\norm{B}^2)$.
	\item $\norm{q_i(A + B) - \big[ q_i(A) + \sum_{j \neq i}{\frac{q_j(A)^T B q_i(A)}{\lambda_i(A) - \lambda_j(A)}q_j(A)}\big]} = O(\norm{B}^2)$.
\end{enumerate}
\end{lemma}

Lemma~\ref{lemma:pert2} gives rise to an improved error bound due to the extra correction term. Using \eqref{eqn:coreection_explained}, the rank-one update~\eqref{eq:proxy} followed by the correction term obtained by~\eqref{lemma:pert2} induces an error of order $\norm{\Delta L - (\Delta L)_1}^2 =\sigma_2^2(\Delta L)$, and by Theorem~\ref{thm:rankone}, we get that 
\[ \sigma_2^2(\Delta L) = O\left(\frac{1}{k^2}\right) . \] 
The perturbation correction is embedded in our method as described in Algorithm~\ref{alg:extension_algo_corr}. The complexity of Algorithm~\ref{alg:extension_algo_corr} is discussed in detail on Appendix~\ref{sec:comp_lag}.

\begin{algorithm}[ht]
\caption{Out-of-sample extension of the graph Laplacian}
\label{alg:extension_algo_corr}
\begin{algorithmic}[1]
\REQUIRE  The original graph Laplacian $L_0$ and its top $m$ eigenpairs $ \left\lbrace  \left( \lambda_i, q_i \right) \right\rbrace_{i=1}^m$. \\ A new sample point $x_0$.
\ENSURE Approximate top eigenpairs $ \left\lbrace  \left( \widehat{t_i},\widehat{p_i} \right) \right\rbrace_{i=1}^m$
\STATE $L_1 \gets $ the graph Laplacian of $\mathcal{X} \cup \{x_0\}$ \label{alg2:line1}
\STATE $\Delta L \gets L_1 - L_0$
\STATE $ \rho \gets \lambda_1(\Delta L)$   \label{alg2:line3}
\STATE $ v \gets q_1 ( \Delta L)$                 \label{alg2:line4}
\STATE Approximate $\left\lbrace  \left( \widetilde{t_i},\widetilde{p_i} \right) \right\rbrace_{i=1}^m $ using Algorithm~\ref{alg:trunc} with input $\left( \left\lbrace  \left( \lambda_i, q_i \right) \right\rbrace_{i=1}^m, \rho, v \right)$
\label{alg2:line5}
\STATE $C \gets L_1 - (L_0 + \rho v v^T)$
\label{alg2:line6}
\FORALL[perturbation correction]{ $i = 1 ... m $ } \label{alg2:line7}
\STATE  $ \widehat{t_i} \gets \widetilde{t_i} + \widetilde{p_i^T} C \widetilde{p_i}$
\STATE $ \widehat{p_i} \gets  \widetilde{p_i} + \sum_{j\neq i} { \frac{\widetilde{p_j}^TC\widetilde{p_i}}{\widetilde{t_i} - \widetilde{t_j}} \widetilde{p_j} }$ \label{alg2:line9}
\ENDFOR \label{alg2:line10}
\end{algorithmic}
\end{algorithm}

\section{Numerical examples} \label{sec:numeric}

In this section, we provide various numerical examples to demonstrate empirically the theory developed in the previous sections. We use both synthetic datasets as well as real-world datasets. We begin by providing several numerical examples for the rank-one update formulas of Section~\ref{sec:rank_one_update}. These examples demonstrate the high accuracy of the methods, as well as their runtime efficiency. We continue by providing numerical examples for Section~\ref{sec:updating_problem}, showing numerically that inserting a new vertex to the graph Laplacian is almost a rank-one update to its matrix. We proceed by applying our algorithm for updating the eigenvalues and eigenvectors of the graph Laplacian to real-world data and measure the accuracy of our approach compared to other methods. All experiments were performed on an Intel i7 desktop with $8$GB of RAM. All algorithms were implemented in MATLAB. The code to reproduce the examples is available at https://github.com/roymitz/rank-one-update.

\subsection{Truncated formulas for rank-one update (Section~\ref{sec:rank_one_update})}

We start with a synthetic example to demonstrate empirically the use of the truncated secular equation and eigenvectors formula for the rank-one update problem. We generate a random symmetric matrix $A \in \mathbb{R}^{n \times n}$ with $n = 1000$ and $m = 10$ known leading eigenvalues whose magnitude is $O(1)$, together with their corresponding eigenvectors. The rest of the eigenvalues are unknown to the algorithm and are drawn from a normal distribution with mean $\hat{\mu}$ and standard deviation of $\sigma = 0.0001$. For the update, we use a random perturbation vector $v$. The goal is to recover the $m$ top eigenpairs of $A + vv^T$ for various values of $\hat{\mu}$. As a rough estimate for the unknown eigenvalues, our parameter $\mu$ is chosen to be either $\mu = 0$ or $\mu = \mu_\ast$ \eqref{eq:mu_opt}. 

The results are shown in Table~\ref{tbl:synt_errors_mu} and Figure~\ref{fig:sec_err_decay}. For the approximate eigenvalues, we measure the absolute errors of the first order method \eqref{eq:TSE} and the second order method \eqref{eq:CTSE}, for the two different choices of $\mu$. Note that for $\mu = \mu_\ast$, the two estimations are identical and thus appear in the same column of the table. For the eigenvectors, the norm of the approximation error is presented, for the two different methods: first order of \eqref{eqn:TEF} and second order of \eqref{eq:CTEF} using the two different choices of $\mu$.

According to Section~\ref{sec:rank_one_update}, for the above setting we expect the case $\mu = 0$ to yield errors of magnitude $O(\hat{\mu})$ for the first order approximations, and of magnitude $O(\hat{\mu}^2)$ for the second order approximations. For $\mu = \mu_\ast$ we expect errors independent of $\hat{\mu}$. This may be observed in Table~\ref{tbl:synt_errors_mu}, but is even clearer in Figure~\ref{fig:sec_err_decay}, where we have a  line with zero slope for $\mu_\ast$ (error is independent on $\hat{\mu}$), a line with slope equal to one for $\mu = 0$ and first order approximation (linear error decay), and a line with slope equal to two for $\mu = 0$ and second order approximation (quadratic error decay).

\begin{table}
 \begin{adjustbox}{max width=\textwidth}
\centering
\begin{tabular}{|c|c|c|c|c|c|c|c|}
\hline
        & \multicolumn{3}{c|}{Eigenvalues}                                                                                                                                                                              & \multicolumn{4}{c|}{Eigenvectors}                                                                                                                                                                                                                                                       \\ \hline
$\hat{\mu}$ & \begin{tabular}[c]{@{}c@{}}first order\\ $\mu = 0$\end{tabular} & \begin{tabular}[c]{@{}c@{}}second order\\ $\mu = 0$\end{tabular} & \begin{tabular}[c]{@{}c@{}}first order\\second order \\ $\mu = \mu_\ast$\end{tabular} & \begin{tabular}[c]{@{}c@{}}first order\\ $\mu = 0$\end{tabular} & \begin{tabular}[c]{@{}c@{}}second order\\ $\mu = 0$\end{tabular} & \begin{tabular}[c]{@{}c@{}}first order\\ $\mu = \mu_\ast$ \end{tabular} & \begin{tabular}[c]{@{}c@{}}second order\\ $\mu = \mu_\ast$\end{tabular} \\ \hline
1e-00   & 8.79e-02                                                        & 3.82e-02                                                         & 9.22e-10                                                                 & 1.79e-01                                                        & 1.70e-01                                                                & 3.45e-05                                                         & 5.25e-08                                                                 \\ \hline
1e-01   & 4.20e-03                                                        & 4.24e-04                                                         & 4.42e-10                                                                 & 1.26e-02                                                        & 7.90e-03                                                                & 9.68e-06                                                         & 8.27e-09                                                                 \\ \hline
1e-02   & 3.08e-04                                                        & 2.77e-06                                                         & 2.72e-10                                                                 & 7.83e-04                                                        & 9.72e-05                                                                & 8.28e-06                                                         & 9.61e-09                                                                 \\ \hline
1e-03   & 3.00e-05                                                        & 2.68e-08                                                         & 2.61e-10                                                                 & 7.66e-05                                                        & 1.00e-06                                                                & 8.20e-06                                                         & 9.88e-09                                                                 \\ \hline
1e-04   & 3.12e-06                                                        & 5.83e-10                                                         & 2.95e-10                                                                 & 1.17e-05                                                        & 2.21e-08                                                                & 8.72e-06                                                         & 1.12e-08                                                                 \\ \hline
\end{tabular}
\end{adjustbox}
\caption{Absolute errors for the synthetic example with $n = 1000, m = 10$, with unknown eigenvalues distributed normally with mean $\hat{\mu}$ and standard deviation $\sigma = 0.0001$.}
\label{tbl:synt_errors_mu}
\end{table}

\begin{figure}
    \begin{center}
        \begin{subfigure}[b]{0.45\textwidth}
        \includegraphics[width=\textwidth]{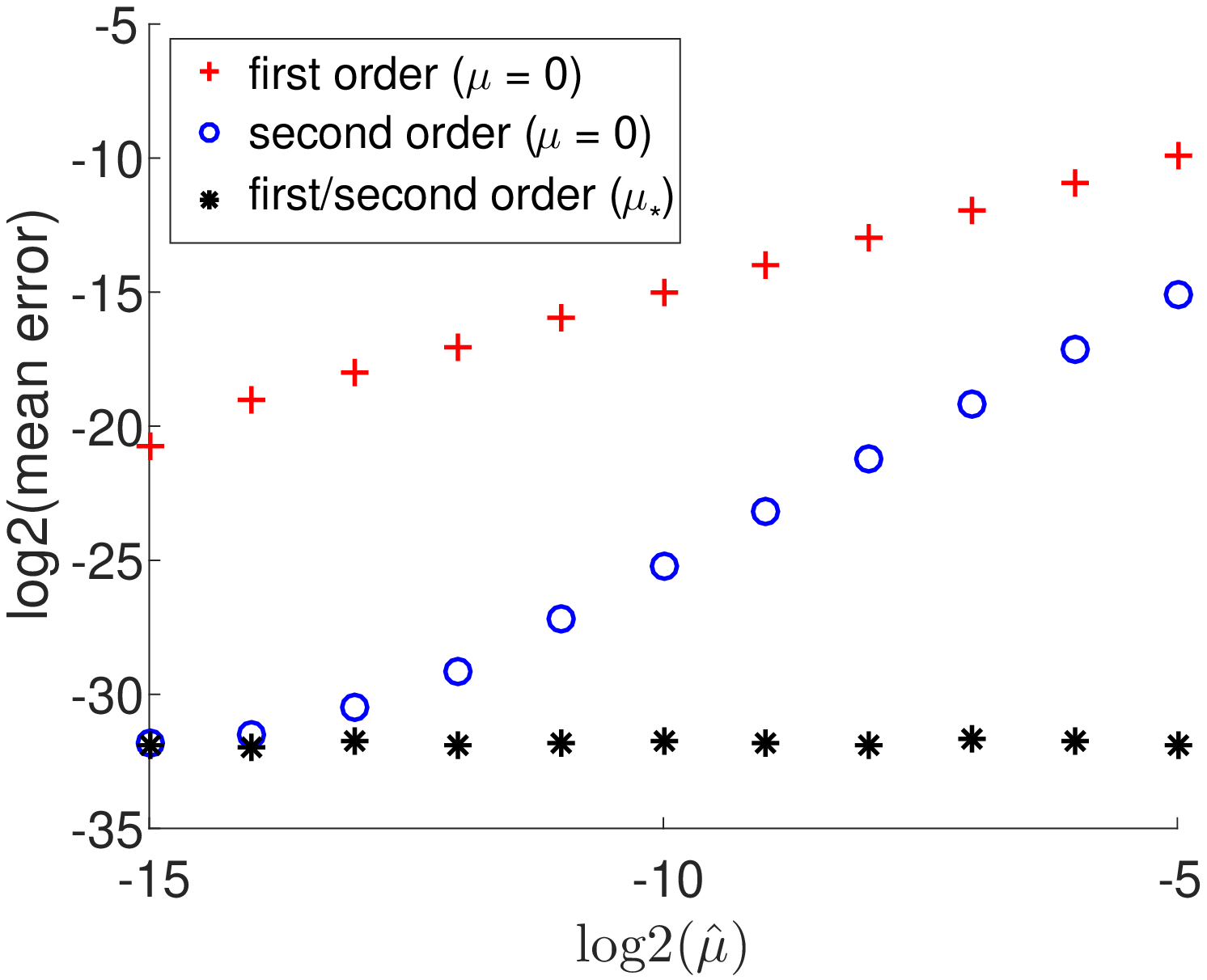}
        \caption{Eigenvalues absolute error}
    \end{subfigure} \quad
    \begin{subfigure}[b]{0.45\textwidth}
        \includegraphics[width=\textwidth]{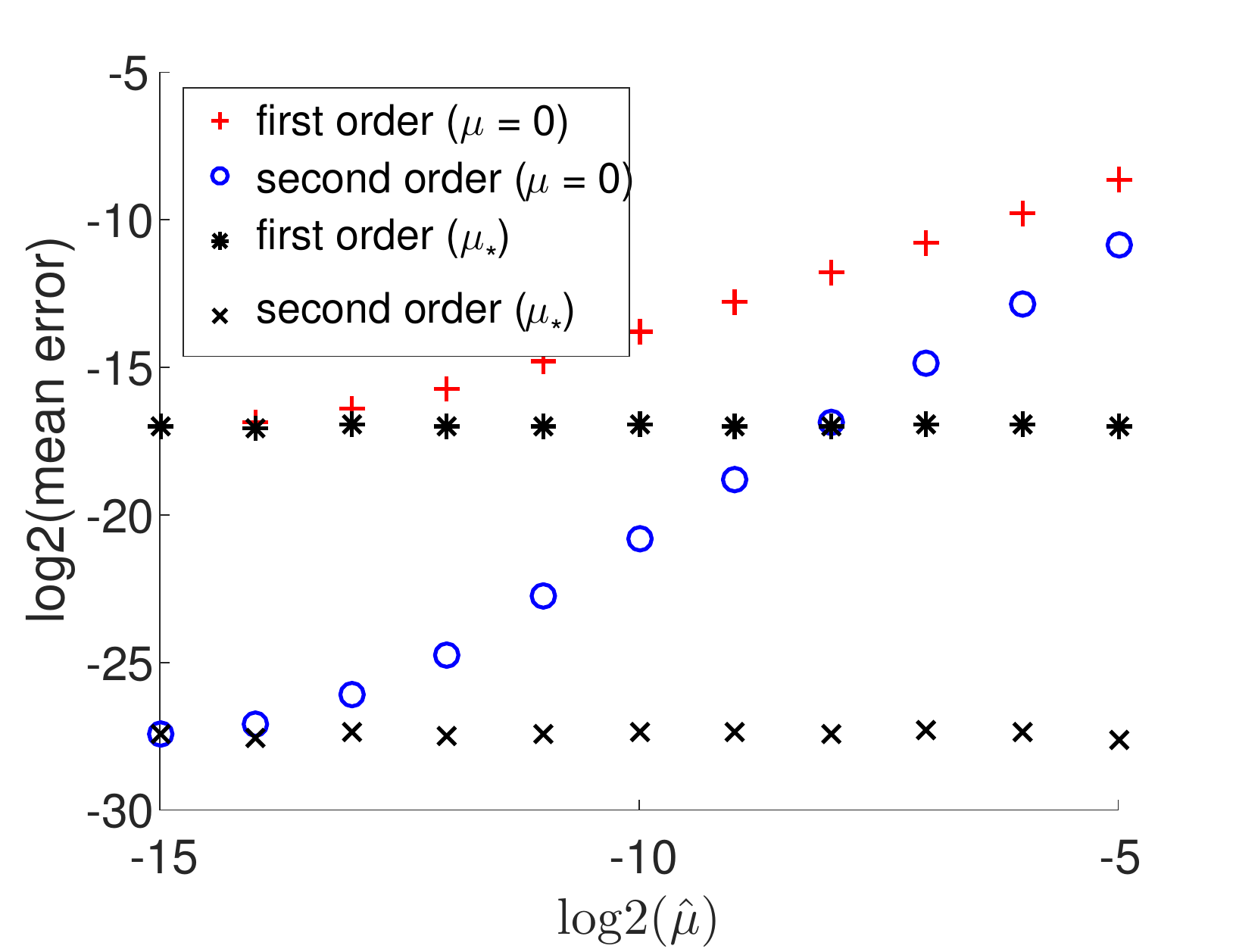}
        \caption{Eigenvectors absolute error}
    \end{subfigure}    
        \end{center}
\caption{Plot of $\log_2$-absolute error as a function of $\log_2\hat{\mu}$ for the synthetic example with $n = 1000, m = 10$, with unknown eigenvalues normally distributed with mean $\hat{\mu}$ and standard deviation $\sigma = 0.0001$. We can notice the three main trends: the error when using $\mu_\ast$ is independent of $\hat{\mu}$, the error when using 
$\mu = 0$ and the first order approximation decay\ linearly, and the error when using $\mu = 0$ and the second order approximation decay quadratically.}
     \label{fig:sec_err_decay}
\end{figure}

The next example demonstrates the mean running time of Algorithm~\ref{alg:trunc} for $10$ independent runs, compared to MATLAB's function \texttt{eigs(L1, m)} for calculating the leading $m$ eigenvalues and eigenvectors. The setting of the example is as follows. A symmetric sparse random matrix with $O(100 \cdot n)$ non-zero entries and a  sparse random vector $v$ with $O(100)$ entries were generated. We then used two variants of our algorithm to update the eigenpairs: first order approximation with $\mu = 0$ (fastest variant) and second order approximation with $\mu = \mu_\ast$ (slowest variant). Table~\ref{tbl:runtime} demonstrates the dependence of the running time on $n$ and $m$. While MATLAB's algorithm is accurate and ours is only an approximate, we can see that for relatively small values of $n$ our algorithm is more than an order of magnitude faster. Due to the linear dependence on $n$, we can expect this difference to be even more dramatic for larger values of $n$, as witnessed in Figure~\ref{fig:runtime_n}. Additionally, the runtime differences between the two variants of our algorithm are negligible. 

\begin{table}
\centering
\begin{subtable}{.45\textwidth}
\begin{adjustbox}{width={\textwidth},totalheight={\textheight},keepaspectratio}
\begin{tabular}{|c|l|l|l|}
\hline
$n$  & \multicolumn{1}{c|}{MATLAB} & \multicolumn{1}{c|}{\begin{tabular}[c]{@{}c@{}}first order\\ $\mu = 0$\end{tabular}} & \multicolumn{1}{c|}{\begin{tabular}[c]{@{}c@{}}second order\\ $\mu = \mu_\ast$\end{tabular}} \\ \hline
2000 & 0.47 $\pm$ 0.05                     & 0.75 $\pm$ 0.02                                                                      & 0.85 $\pm$ 0.02                                                                               \\ \hline
4000 & 1.71 $\pm$ 0.10                     & 0.79 $\pm$ 0.02                                                                      & 0.92 $\pm$ 0.02                                                                               \\ \hline
8000 & 5.34 $\pm$ 0.05                     & 0.76 $\pm$ 0.03                                                                      & 0.88 $\pm$ 0.02                                                                               \\ \hline
16000 & 17.1 $\pm$ 0.12                     & 0.86 $\pm$ 0.03                                                                      & 0.97 $\pm$ 0.02                                                                               \\ \hline
32000 & 50.6 $\pm$ 0.57                     & 0.97 $\pm$ 0.03                                                                      & 1.11 $\pm$ 0.02                                                                               \\ \hline
64000 & 154 $\pm$ 1.01                     & 1.23 $\pm$ 0.01                                                                      & 1.36 $\pm$ 0.01                                                                               \\ \hline
\end{tabular}
\end{adjustbox}
\caption{Timing is seconds for varying $n$ with a fixed number of eigenpairs, $m = 10$.}
\end{subtable}  
\hspace{1cm}
\begin{subtable}{.45\textwidth}
\begin{adjustbox}{width={\textwidth},totalheight={\textheight},keepaspectratio}
\begin{tabular}{|c|l|l|l|}
\hline
$m$ & \multicolumn{1}{c|}{MATLAB} & \multicolumn{1}{c|}{\begin{tabular}[c]{@{}c@{}}first order\\ $\mu = 0$\end{tabular}} & \multicolumn{1}{c|}{\begin{tabular}[c]{@{}c@{}}second order\\ $\mu = \mu_\ast$\end{tabular}} \\ \hline
50  & 12.4 $\pm$ 0.13                     & 0.45 $\pm$ 0.01                                                                      & 0.51 $\pm$ 0.01                                                                               \\ \hline
100 & 26.5 $\pm$ 0.34                     & 0.95 $\pm$ 0.01                                                                      & 1.08 $\pm$ 0.01                                                                               \\ \hline
200 & 60.5 $\pm$ 0.33                     & 1.98 $\pm$ 0.04                                                                      & 2.21 $\pm$ 0.04                                                                               \\ \hline
400 & 155 $\pm$ 4.05                     & 4.74 $\pm$ 0.20                                                                      & 5.08 $\pm$ 0.42                                                                               \\ \hline
600 & 345 $\pm$ 4.05                     & 7.14 $\pm$ 0.20                                                                      & 8.25 $\pm$ 0.42                                                                               \\ \hline
800 & 542 $\pm$ 4.05                     & 10.9 $\pm$ 0.20                                                                      & 11.7 $\pm$ 0.42                                                                               \\ \hline
\end{tabular}
\end{adjustbox}
\caption{Timing is seconds for a varying number of eigenpairs $m$ with a fixed matrix size $n = 20,000$.}
\end{subtable}
\caption{Performance measurements.}
\label{tbl:runtime}
\end{table}

\begin{figure}
    \centering
    \includegraphics[width=0.5\textwidth]{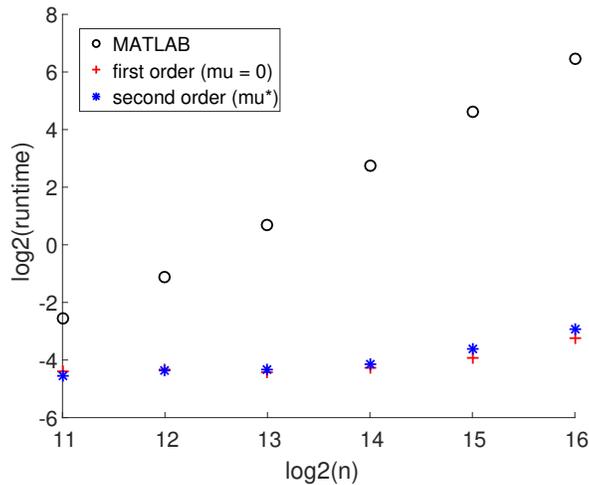}
    \caption{Plot of $\log_2$-runtime as a function of $\log_2n$, for matrices of size $n$. We can see that the runtime difference between our algorithms and MATLAB's \texttt{eigs} increases with $n$.}
\label{fig:runtime_n}
\end{figure}

\subsection{Updating the graph Laplacian (Section~\ref{sec:updating_problem})}
	
We provide several examples using three real-world datasets to demonstrate the update of the symmetric graph Laplacian. The datasets are described in Table~\ref{tbl:datasets}.

\begin{table}
		\begin{center}
    	\begin{tabular}  { | l | l | l | p{7cm} |}
    	\hline
    	Name & Samples & Attributes & Description \\ \hline
    	MNIST & 60,000 & 784 & Grey scale images of handwritten digits between $0$ and $9$ \\ \hline
    	Poker Hand & 25,000 & 10 & Each record is a hand consisting of five playing cards drawn from a standard deck of 52 cards\\ \hline
    	Yeast & 1484 & 8 &  Information about a set of yeast cells \\ \hline
    	\end{tabular} \end{center} \caption{Real-world datasets.}
    \label{tbl:datasets}
	\end{table}
	
In the first example, we demonstrate that inserting a new vertex to the graph Laplacian is almost rank-one, as suggested by Theorem~\ref{thm:rankone}. In this example, for each dataset, we first randomly select a subset of it, and construct the symmetric graph Laplacian $L_0$ of the selected subset, leaving the first vertex out. Then, we connect this vertex to the graph, which results in a new graph Laplacian $L_1$. Finally we compute the first and second singular values of $\Delta L = L_1 - L_0$. We repeat this experiment 10 times, each time with a different random subset of the data. The results of the mean magnitude of the singular values are shown in Table~\ref{tbl:sigma1} for various datasets and values of $k$. Clearly, one can observe, as predicted by the theory, that the first singular value is very close to $1$, while the second singular value is close to $0$.

\begin{table}
\begin{center}
\begin{tabular}{|l|l|l|l|l|l|l|}
\hline
                     & \multicolumn{2}{l|}{$k=5$} & \multicolumn{2}{l|}{$k=10$} & \multicolumn{2}{l|}{$k=20$} \\ \hline
dataset              & $\sigma_1$    & $\sigma_2$   & $\sigma_1$    & $\sigma_2$   & $\sigma_1$    & $\sigma_2$   \\ \hline
MNIST (5K samples)   & 0.91         & 0.09        & 0.97         & 0.05 & 0.99         & 0.02        \\ \hline
poker (10K samples)  & 0.94         & 0.09        & 0.98         & 0.05 & 0.98         & 0.02        \\ \hline
yeast (1.5K samples) & 0.94         & 0.11       & 0.97         & 0.05        & 0.99         & 0.03        \\ \hline
\end{tabular}
\caption{The two largest singular values of $\Delta L$ for real-world datasets and three different values of $k$. As theory predicts, there is a two orders of magnitude difference between the first and second singular values, indicating that indeed $\Delta L$ is close to being rank-one.}
\label{tbl:sigma1}
\end{center}
\end{table}

Next, we demonstrate empirically the dependence of the singular values on $k$. Specifically, Theorem~\ref{thm:rankone} implies that up to a constant
\[ \log(\sigma_1(\Delta L) - 1) =  \log(\sigma_i(\Delta L)) = -k , \quad 2 \le i \le n. \]
Thus the log of the singular values is expected to be linear in $k$ with slope that equals to 1. This is demonstrated for the poker dataset in Figure~\ref{fig:mnist_sings}. Similar results were obtained for the other datasets as well. 

\begin{figure}[t]
    \centering
    \begin{subfigure}[b]{0.4\textwidth}
        \includegraphics[width=\textwidth]{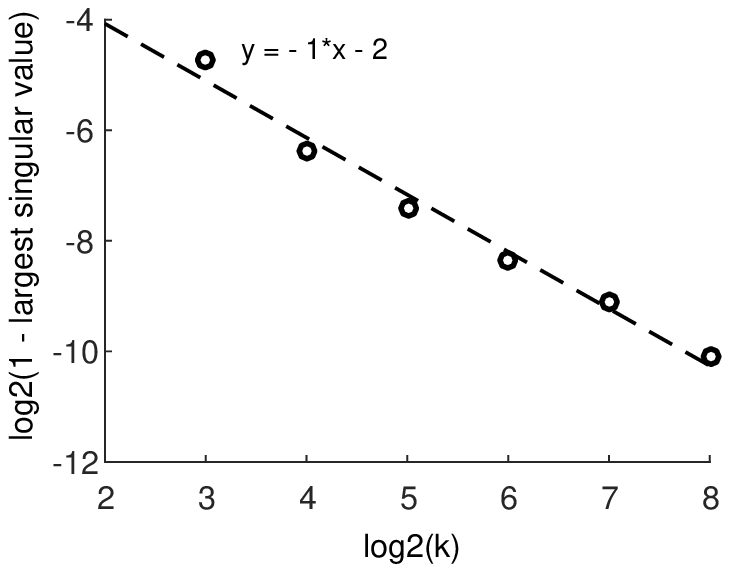}
        \caption{$1 - \sigma_1(\Delta L) = O\big(\frac{1}{k}\big)$. The dependence of $\operatorname{log}_2(1 - \sigma_1)$ on $k$ is linear with slope $(-1)$.}
    \end{subfigure} 
    \qquad \quad
    \begin{subfigure}[b]{0.42\textwidth}
        \includegraphics[width=\textwidth]{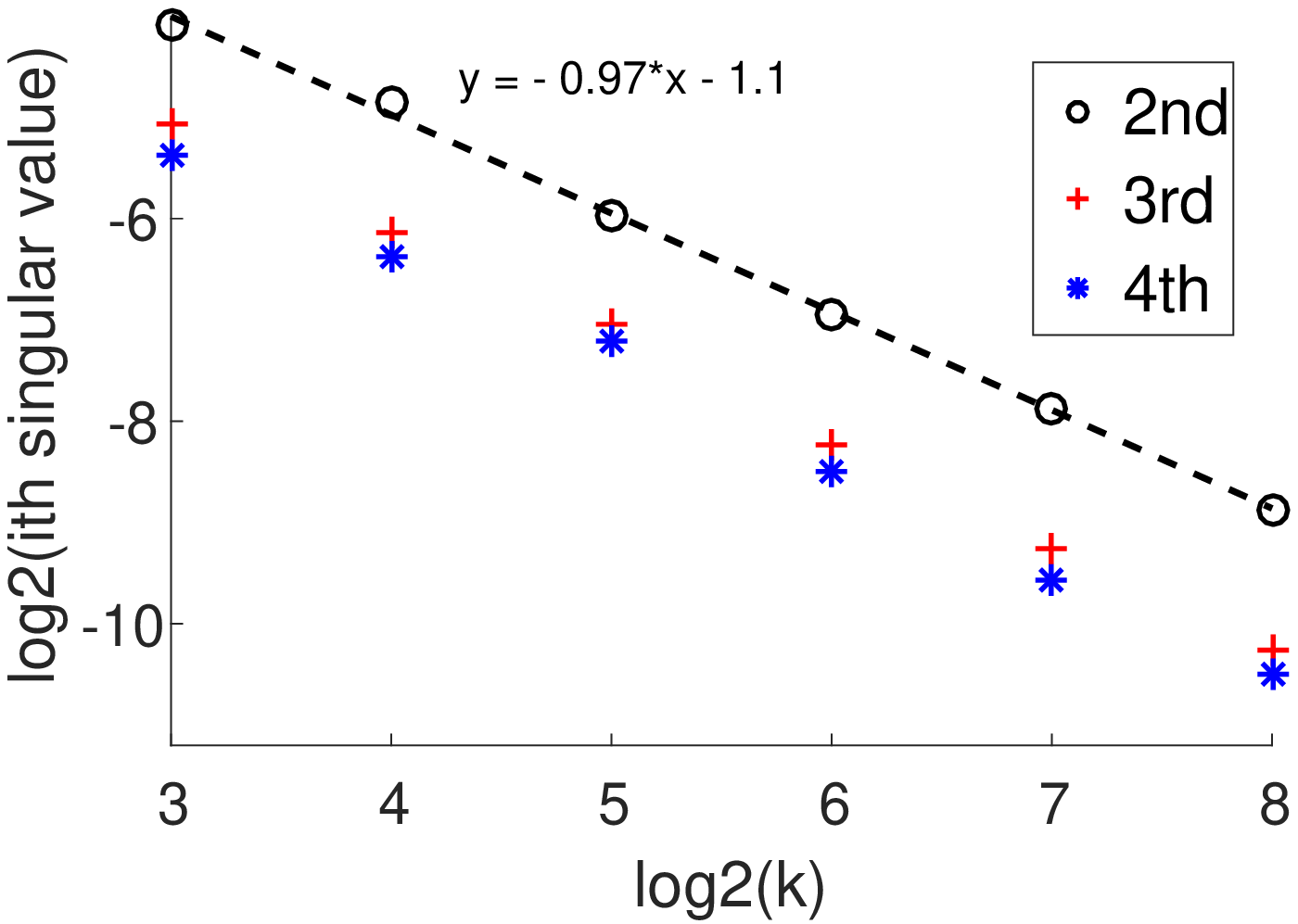}
        \caption{$\sigma_i(\Delta L) = O\big(\frac{1}{k}\big)$  for $i=2,3,4$. The dependence of $\operatorname{log}_2(\sigma_i)$ on $k$ is linear with slope $(-1)$.}
    \end{subfigure}   
    \caption{Demonstration of Theorem~\ref{thm:rankone} for the poker dataset with $n=10,000$.}\label{fig:mnist_sings}
\end{figure}

In the second example, we perform out-of-sample extension using several methods and compare their accuracy. As a benchmark for the eigenvectors extension, we use the Nystr{\"o}m method \cite{bengio2004out}, which is a widely used method for this task. Additionally, we use the naive approach of merely having the old eigenvalues and eigenvectors as approximations to the new ones. Regarding our methods, we use both the first order (\eqref{eq:TSE}, \eqref{eqn:TEF}) and second order (\eqref{eq:CTSE},\eqref{eq:CTEF}) approximations described in Section~\ref{sec:rank_one_update}. For our methods, we also apply the perturbation correction described in Algorithm~\ref{alg:extension_algo_corr}. To compare the performance of the different algorithms, we measure the angles between the true eigenvectors and their approximations, and report the maximal angle out of the $m$ angles calculated. The results reported are the mean error of $10$ independent experiments, that is, picking randomly a vertex for the out-of-sample extension in each experiment. The full comparison between the described methods is given in Table~\ref{tbl:comparison}, where for each dataset we also mention the parameter $\varepsilon$ of the width of the Gaussian (see \eqref{eqn:guass_ker}) that we used for constructing the graph Laplacian. On the second column of Table~\ref{tbl:comparison} is the absolute error of the eigenvalues for each method except the Nystr{\"o}m method which we use only to extend the eigenvectors. On the third column is the absolute error of the eigenvalues after performing the perturbation correction. The fifth and sixth columns present the error of the eigenvectors estimation, before and after perturbation correction, respectively. 

We can see that our methods outperform the other approaches. As expected, the second order approximations using $\mu_{*}$ present the best performance and is marked in bold.

\begin{table}
\centering
\begin{tabular}{|l|l|l|l|l|}
\hline
MNIST                                     & Eigenvalues       & \begin{tabular}[c]{@{}l@{}}Eigenvalues\\ (after correction)\end{tabular} & Eigenvectors      & \begin{tabular}[c]{@{}l@{}}Eigenvectors\\ (after correction)\end{tabular} \\ \hline
No update                                 & 7.85e-05          & -                                                          & 2.83\degree         & - \\ \hline
Nystr{\"o}m                                   & -          & -                                                          & 1.56\degree          & - \\ \hline
First order ($\mu = 0$)                   & 5.09e-05 & 7.73e-06                                                          & 1.00\degree          & 0.84\degree \\ \hline
\textbf{Second order ($\mu = \mu_\ast$)} & \textbf{5.06e-05} & \textbf{7.70e-06}                                                 & \textbf{1.00\degree} & \textbf{0.82\degree}                                                  \\ \hline
\end{tabular}

\bigskip

\begin{tabular}{|l|l|l|l|l|}
\hline
Poker                                     & Eigenvalues       & \begin{tabular}[c]{@{}l@{}}Eigenvalues\\ (after correction)\end{tabular} & Eigenvectors      & \begin{tabular}[c]{@{}l@{}}Eigenvectors\\ (after correction)\end{tabular} \\ \hline
No update                                 & 1.97e-05 & -                                                          & 3.04\degree          & -                                                           \\ \hline
Nystr{\"o}m                                   & -          & -                                                          & 2.71\degree          & - \\ \hline
First order ($\mu = 0$)                   & 1.44e-05          & 6.18e-06                                                          & 1.89\degree          & 0.95\degree \\ \hline
\textbf{Second order ($\mu = \mu_\ast$)} & \textbf{1.43se-05} & \textbf{6.15e-06}                                                 & \textbf{1.88\degree} & \textbf{0.94\degree}                                                  \\ \hline
\end{tabular}

\bigskip

\begin{tabular}{|l|l|l|l|l|}
\hline
Yeast                                     & Eigenvalues       & \begin{tabular}[c]{@{}l@{}}Eigenvalues\\ (after correction)\end{tabular} & Eigenvectors      & \begin{tabular}[c]{@{}l@{}}Eigenvectors\\ (after correction)\end{tabular} \\ \hline
No update                                 & 1.58e-04          & -                                                          & 2.40\degree          & - \\ \hline
Nystr{\"o}m                                                                     & -          & - & 0.65\degree          & -                                                           \\ \hline
First order ($\mu = 0$)                   & 1.11e-04          & 1.78e-06                                                          & 0.42\degree        & 0.35\degree \\ \hline
\textbf{Second order ($\mu = \mu_\ast$)} & \textbf{1.11e-04}          & \textbf{1.78e-06                                                         } & \textbf{0.41\degree}          & \textbf{0.33\degree} \\ \hline 
\end{tabular}

\caption{Error comparison for three datasets: MNIST $(n = 1000, m = 5, k = 10, \varepsilon = 100)$, poker $(n = 3000, m = 5, k = 100, \varepsilon = 100)$ and yeast $(n = 1400, m = 5, k = 100, \varepsilon = 100)$. Best performance is marked in bold.}
 \label{tbl:comparison}
\end{table}

In the last example, we demonstrate a practical rather than a numerical advantage of our method. Starting with $1500$ random samples from the MNIST dataset, we split this set into a train set consisting of $1000$ samples and a test set consisting of the remaining $500$ samples. Each point in the train set is in $\mathbb{R}^{784}$. We then embed the train set samples in $\mathbb{R}^{10}$ using Laplacian eigenmaps~\cite{belkin2003laplacian} with parameters $k=10$ and $\varepsilon = 100$ for constructing the graph Laplacian. For each sample in the test set, we perform an out-of-sample-extension using four different methods: recalculation of the new embedding (which is the optimal, expensive method), no update where the test points are embedded naively to the origin, Nystr{\"o}m method, and our method. We then train a 15-NN classifier on the embedded vectors of the train set. We use this classifier to label the given test sample and compare it to the true label. Table~\ref{tbl:ml_comparison} summarizes the accuracy of each extension method on the test set. One can see that our method performs considerably better than the other approaches, and its performance is very close to the best possible results obtained by the method of recalculating the entire embedding.

\begin{table}
\centering
\begin{tabular}{|l|l|}
\hline
Method                       & Accuracy \\ \hline
Recalculation (Optimal)                &68\%     \\ \hline
No update                    & 12\%     \\ \hline
Nystr{\"o}m & 58\%     \\ \hline
Our method                   & 67\%     \\ \hline
\end{tabular}
\caption{Accuracy for out-of-sample extension of the MNIST dataset for several methods.}
\label{tbl:ml_comparison}
\end{table}

\section{Conclusions} \label{sec:conclusions}

In this paper, we proposed an approximation algorithm for the rank-one update of a symmetric matrix when only part of its spectrum is known.  We provided error bounds for our algorithm and showed that they are independent of the number of unknown eigenvalues. As implied both by theory and numerical examples, our algorithm performs best when the unknown eigenvalues are clustered (i.e., close to each other). On the other hand, numerical evidence shows that when the unknown eigenvalues are not clustered, the results may deteriorate but are still no worse than neglecting the unknown eigenvalues (low rank approximation). As a possible application, we proposed the out-of-sample extension of the graph Laplacian matrix, and demonstrated that our method provides superior results. 

%
\section*{Acknowledgments}
This research was supported by THE ISRAEL SCIENCE FOUNDATION grant No. 578/14, by Award Number R01GM090200 from the NIGMS, and by the European Research Council (ERC) under the European Union’s Horizon 2020 research and innovation programme (grant agreement 723991 - CRYOMATH). NS was partially supported by the Moore Foundation Data-Driven Discovery Investigator Award.
%

\bibliographystyle{plain}
\bibliography{OOS_bib}


\appendix
\label{appendix}  


\section{Complementary materials} \label{app:vf}

\subsection{Proof of Theorem~\ref{thm:extendweyl}} \label{sec:ProofOfWeyl2thm}
\begin{proof}     
Denote by $\vect(X)$ the $n \times n$ matrix $X$ rearranged as a vector in $\mathbb{R}^{n^2}$. By assumption, the diagonal entries of $S$ are different of each other, and thus the singular values of $S$ are distinct. When the singular values of a matrix are distinct, they are analytic functions of its entries in some compact $\eta$-neighbourhood $S^*$ of $\vect(S)$ and can be expanded in a Taylor expansion \cite{magnus1985differentiating}. Let $E$ be a matrix so that $\vect(S+E) \in S^*$, i.e., $\|\vect(S + E) - \vect(S)\| = \|\vect(E)\| < \eta$. We use the result in \cite{papadopoulo2000estimating}, stating that if $A = UDV^T$ is the SVD of $A$, then $\frac{\partial \sigma_k(A)}{\partial a_{ij}} = u_{ik}v_{jk}$. Expanding the singular value function to a first order Taylor polynomial yields
\begin{equation} \label{eqn:TaylorSVD}
\begin{split}
\sigma_i(S+E) & = \sigma_i(S) + \nabla\sigma_i(S) \cdot \vect(E) + R \\
& = \sigma_i(S) + \sum_{n,m}\frac{\partial \sigma_i(S)}{\partial s_{nm}} \cdot e_{nm} + R = \sigma_i(S) + \sum_{n,m} u_{ni}v_{mi}\cdot e_{nm} + R ,\\
\end{split}
\end{equation}
with $u_i$ and $v_i$ being the left and right singular vectors of $S$ respectively.
Denote by $|X|$ the matrix whose entries are the absolute values of the entries of the matrix $X$ and by $P_S(X)$ the matrix resulting by flipping the sign of the entries of column $i$ in the matrix $X$ by the sign of $S_{ii}$ for all $1 \leq i \leq n$. Then, the SVD of the matrix $S$ is $S = I|S|P_S(I)^T$, where $I$ is the identity matrix. Equation~\eqref{eqn:TaylorSVD} is then reduced to
\begin{equation}
\sigma_i(S+E) = \sigma_i(S) +  u_{ii}v_{ii} e_{ii} + R  = \sigma_i(S) +  \operatorname{sign}(S_{ii})e_{ii} + R .\\
\end{equation}
The remainder $R$ has the form $R = \frac12 \vect(E)^T H(\sigma_i(Z)) \vect(E)$ with $H(\sigma_i(Z))$ being the Hessian matrix evaluated for a matrix $Z$ that lies between $S$ and $S + E$, that is $Z = S + cE \in S^*$ for $c \in (0,1)$. Therefore, by Cauchy-Schwartz inequality we have
\begin{equation}
|R| = \frac12 |\vect(E)^TH(\sigma_i(Z))\vect(E)|  \le \frac12 \norm{\vect(E)}\norm{H(\sigma_i(Z))}_2\norm{\vect(E)}. 
\end{equation}

The entries of $H(\sigma_i(X))$ are the second order directional derivatives of $\sigma_i(X)$. Since $\sigma_i(X)$ is an analytic function in $S^*$, its second order derivatives at any direction are continuous functions. Since the $\| \cdot \|_2$ norm is a continuous function of the matrix entries, we conclude that $\norm{H(\sigma_i(X))}_2$ is a continuous function in $S^*$. By the compactness of $S^*$ and the boundedness theorem, there exists $c_H>0$ such that  $\norm{H(\sigma_i(X))}_2 < c_H$ for all $X \in S^*$. Finally we conclude that
\begin{equation}
|R| \leq \frac12 c_H \|E \|^2_F  .
\end{equation}
\end{proof}

\subsection{Proof of Lemma~\ref{lem:boundsize}} \label{sec:ProofOfLemmaBoundSize}
\begin{proof}

Let $w_{pq}$ be the weight on the edge connecting vertices $p$ and $q$ in the graph. We examine the entry $(\Delta L)_{ij}$. If neither of the vertices $i$ nor $j$ were connected to the new vertex, then $\Delta L_{ij} = 0$. Otherwise, assume that the new vertex was connected to vertex $i$ but was not connected to vertex $j$. Denote by $w_{i1} = w_{1i}$ the weight on the edge connecting the new vertex and vertex $i$. Let $\alpha = \sum_p w_{pj}$ (the sum of the $j$-th column of $L_0$) and $\beta = \sum_p w_{ip}$ (the sum of the $i$-th row of $L_0$). It follows that the $i$-th row of $L_1$ is normalized by $\sqrt{w_{i1} + \beta}$. Denote by $l^0_{ij}$ and $l^1_{ij}$ the $(i,j)$ entry of $L_0$ and $L_1$, respectively. Thus,
\begin{equation} \label{eqn:deltaL_ij}
\begin{split}
|(\Delta L)_{ij}|  = \abs{\ell^0_{ij} - \ell^1_{ij}}  & = \frac{w_{ij}}{\sqrt{\alpha} \cdot \sqrt{\beta}} - \frac{w_{ij}}{\sqrt{\alpha} \cdot \sqrt{w_{i1} + \beta}} \\
& = \frac{w_{ij}}{\sqrt{\alpha}} \cdot \Bigg( \frac{1}{\sqrt{\beta}} - \frac{1}{\sqrt{w_{i1} + \beta}} \Bigg)  \\
& = \frac{w_{ij}}{\sqrt{\alpha}} \cdot \Bigg( \frac{\sqrt{w_{i1} + \beta} - \sqrt{\beta}}{\sqrt{\beta} \cdot \sqrt{w_{i1} + \beta}} \Bigg)  = \frac{w_{ij}}{\sqrt{\alpha}} \cdot \frac{w_{i1}}{\big(w_{i1} + \beta\big)\sqrt{\beta} + \beta\sqrt{w_{i1} + \beta}},
\end{split} 
\end{equation}
where~\eqref{eqn:deltaL_ij} is derived by multiplying by $\frac{\sqrt{w_{i1} + \beta} + \sqrt{\beta}}{\sqrt{w_{i1} + \beta} + \sqrt{\beta}}$. Using the bounds on the graph Laplacian entries in~\eqref{eq:w_bounds} and by~\eqref{eq:sum_gl}, we get that $\alpha,\beta > kM\delta$ and finally
\begin{equation}
\begin{split}
|(\Delta L)_{ij}| & < \frac{w_{ij}}{\sqrt{kM\delta}} \cdot \Bigg( \frac{w_{i1}}{(M\delta + kM\delta) \sqrt{kM\delta} + kM\delta \sqrt{M\delta + kM\delta}} \Bigg) \\
                      & < \frac{g(0) + M \delta}{\sqrt{kM\delta}} \cdot \Bigg( \frac{g(0) + M \delta}{kM\delta \sqrt{kM\delta}  + kM\delta \sqrt{kM\delta}} \Bigg) = \frac{(g(0) + M \delta)^2}{2(M\delta)^2 k^2} = \frac{c^2}{2k^2} .
\end{split}
\end{equation}
\noindent
The remaining cases, i.e., when both vertex $i$ and vertex $j$ where connected to the new vertex and where only vertex $j$ was connected to the new vertex are analogous.

\end{proof}

\section{Complexity analysis of the algorithms} \label{app:sec_complexity}

\subsection{Analysis of Algorithm~\ref{alg:trunc}} \label{sec:analysis_1}

In line~\ref{alg1:line1} of Algorithm~\ref{alg:trunc} we determine the parameter $\mu$ defined in Section~\ref{sec:trunc_secular}. For $\mu = 0$ no calculation is needed. For $\mu_{mean}$, we need to sum the diagonal elements of an $n\times n$ matrix which costs $O(n)$ operations. The calculation of $\mu_\ast$ requires the calculation of $s$ of \eqref{eq:s} which costs $O \left( mn + \operatorname{nnz}(A) \right)$ operations. Note that this calculation is required for the second order approximation of the secular equation~\eqref{eq:CTSE}. 

Line \ref{alg1:line2} requires solving a variant of the truncated secular equation, namely \eqref{eq:TSE} or \eqref{eq:CTSE}, and is done by standard solvers such as Newton's method. Indeed, if we consider Newton's method, each iteration consists of summing up to $O(m)$ terms, thus requiring in $O(m)$ operations. We also expect this method to converge in $O(1)$ iterations, as Newton's method has a quadratic convergence rate, resulting in $O(m)$ operations for one eigenvalue. The calculation for all the top $m$ eigenvalues of $A + \rho vv^T$ will thus require $O(m^2)$ operations. For the second order approximation, there is an additional calculation of $s$ which is done only once and costs (as already mentioned above) $O \left( mn + \operatorname{nnz}(A) \right)$ operations. The calculation of $z = Q_m^Tv$ is also done only once so we get a total complexity of $O \left( m^2 + mn \right)$ operations for the first order approximation, and additional $O \left( mn + \operatorname{nnz}(A) \right)$ operations for the second order one.

Line \ref{alg1:line4} evaluates the eigenvectors. We first discuss the evaluation of~\eqref{eq:naive_TEF}. Again, $z = Q_m^Tv$ needs to be calculated only once for all the eigenvectors and costs $O(mn)$ operations to compute. The rest of the formula involves a multiplication of an $n \times m$ dense matrix by an $m \times m$ diagonal matrix which costs $O(mn)$ operations. Finally, a product of an $n \times m$ matrix with a vector of length $m$ costs $O(nm)$ operations. Normalizing the resulting vector requires $O(n)$ operations. Thus, it costs $O(mn)$ operations to compute one eigenvector, and $O(mn + m^2n)$ operations to compute all $m$ eigenvectors.

In~\eqref{eqn:TEF} and \eqref{eq:CTEF} we add the calculation of $r$ which requires $O(mn)$ operations and can be done only once for all eigenvectors. Therefore, asymptotically, formula \eqref{eqn:TEF} has the same complexity as that of \eqref{eq:naive_TEF}. Lastly, for the second order approximation~\eqref{eq:CTEF}, we add the calculation of $Ar$ which again can be done only once and costs $O(\operatorname{nnz}(A))$ operations.  Thus the calculation of one eigenvector requires $O(\operatorname{nnz}(A) + mn)$ operations and the calculation of all $m$ eigenvectors requires $ O(\operatorname{nnz}(A) + m^2n)$ operations.

\subsection{Analysis of Algorithm~\ref{alg:extension_algo_corr}} \label{sec:comp_lag}

Line~\ref{alg2:line1} of Algorithm~\ref{alg:extension_algo_corr} constructs the graph Laplacian $L_1$. This step involves finding the vertices in the $\delta$-neighbourhood of the new vertex, which requires $O(n)$ operations, followed by updating $O(k^2)$ entries. This gives a total of $O(n + k^2)$ operations for this line, which is independent of the extension scheme.  

Calculating the rank-one update in lines~\ref{alg2:line3}--\ref{alg2:line4} requires applying the power method~\cite[Chapter 8]{golub2012matrix} to $\Delta L$, which is a sparse $n \times n$ matrix with a large spectral gap (Theorem~\ref{thm:rankone}). We can thus expect the power method to converge within a few iterations.
The complexity of applying the rank-one update in line~\ref{alg2:line5} depends on the method used, and costs at most $O(m^2n + \operatorname{nnz}(L_0))$. The calculation of $C$ in~\ref{alg2:line6} costs $\operatorname{max}\big\{\operatorname{nnz}(L_0), \operatorname{nnz}(L_1), \operatorname{nnz}(v)^2\big\}$.

We now analyse the perturbation correction in lines \ref{alg2:line7}--\ref{alg2:line10}. For one eigenvalue, the correction consists of a multiplication  of the form $q^TCq = q^T(Cq)$, which requires $O(\operatorname{nnz}(C))$ operations for $Cq$ and $O(n)$ operations for $q^T(Cq)$. Thus, the total number of arithmetic operations for one eigenvalue is $O(n + \operatorname{nnz}(C))$, and for correcting all the eigenvalues is $O(mn + m\operatorname{nnz}(C))$.

For one eigenvector, we sum up $m$ elements of the form $\frac{q^TCq}{\lambda}q$. Based on the analysis above, each element requires $O(n + \operatorname{nnz}(C))$ operations, resulting in $O(mn + m\operatorname{nnz}(C))$ operations for one eigenvector and $O(m^2n + m^2\operatorname{nnz}(C))$ operations for all eigenvectors. 

The complete complexity analysis is summarized in Table~\ref{tbl:complexity}.

\begin{table}
\centering
\begin{tabular}{|c|l|l|}
\hline
                              & \multicolumn{1}{c|}{Calculation} & \multicolumn{1}{c|}{Complexity}                                 \\ \hline
\multirow{3}{*}{$\mu$}        & $0$                            & $O(1)$                                                          \\ \cline{2-3} 
                              & $\mu_\ast$                    & $O(mn + \operatorname{nnz}(A))$                                 \\ \cline{2-3} 
                              & $\mu_{mean}$                   & $O(n)$                                                          \\ \hline
\multirow{3}{*}{Eigenvalues}  & First order                    & $O(m^2)$                              \\ \cline{2-3} 
                              & Second order                   & $O(mn + m^2 + \operatorname{nnz}(A))$ \\ \cline{2-3} 
                              & Perturbation Correction        & $O(mn + m \operatorname{nnz}(C))$                               \\ \hline
\multirow{3}{*}{Eigenvectors} & First order                    & $O(m^2n)$                              \\ \cline{2-3} 
                              & Second order                   & $O(m^2n + \operatorname{nnz}(A)$      \\ \cline{2-3} 
                              & Perturbation Correction        & $O(m^2n + m^2 \operatorname{nnz}(C))$                           \\ \hline
\end{tabular}
\caption{Summary of the complexity analysis}
 \label{tbl:complexity}
\end{table}

\section{Loss of orthogonality} \label{app:loss_orth}

The approximate formulas~\eqref{eqn:TEF} and~\eqref{eq:CTEF} will generally not produce an orthogonal set of vectors. Therefore, in its current form, our method may not be applied more than once. In that case, a re-orthogonalization procedure may be used. Examples of such procedures are the QR and polar decompositions. We wish to analyze the quality of the approximation in~\eqref{eqn:TEF} and~\eqref{eq:CTEF} after some re-orthogonalizing procedure has been applied.
Denote $B = A + \rho vv^T$ (see~\eqref{eq:prob}), and let $\widetilde{P} = [\widetilde{p}_1\widetilde{p}_2 \cdots \widetilde{p}_m]$ be the matrix whose columns are the approximated eigenvectors of formulas \eqref{eqn:TEF} or \eqref{eq:CTEF} and $\widetilde{T} = \diag (\widetilde{t}_1,...,\widetilde{t}_m)$ the diagonal matrix of whose diagnoal is the approximated eigenvalues obtained by solving \eqref{eq:TSE} or \eqref{eq:CTSE}. As a measure of the quality of the eigenpairs approximations we will use
\begin{equation}
\norm{B\widetilde{P} - \widetilde{P}\widetilde{T}}_F.
\end{equation}
Let $\bar{P} \in \mathbb{R}^{n \times m}$ be the result of orthogonalizing the columns of $\widetilde{P}$ by any method so that
\begin{equation} \label{reorth}
\widetilde{P} = \bar{P}(I + E)
\end{equation}
for some  $E \in \mathbb{R}^{m \times m}$, and let $G \in \mathbb{R}^{m \times m}$ so that
\begin{equation}
I + G = \widetilde{P}^T\widetilde{P},
\end{equation}
i.e., $G$ determines how much the approximated eigenvectors deviate from being orthonormal. Then we have,
\begin{equation}
I + G = (I + E)^T(I + E).
\end{equation}
By perturbation bound of the Cholesky decomposition~\cite[Theorem 9]{chang1996new} we obtain that if 
\begin{equation} \label{pert_cond}
\norm{G}_F < \frac14
\end{equation}
then
\begin{equation} \label{pert_bound}
\norm{E}_F < 2 \norm{G}_F.
\end{equation}
Inequality \eqref{pert_bound} enables us to bound the affect of re-orthogonalization on the quality of the approximation. By~\eqref{reorth}, we have that
\begin{equation} \label{eq:BP}
B\widetilde{P} = B \bar{P}(I + E)= B \bar{P} + B \bar{P}E
\end{equation}
and
\begin{equation} \label{eq:PT}
\widetilde{P}\widetilde{T} = \bar{P}(I + E)\widetilde{T} = \bar{P}\widetilde{T} + \bar{P}E\widetilde{T}.
\end{equation}
Subtracting~\eqref{eq:PT} from~\eqref{eq:BP}, taking Frobenius norm and applying the triangle inequality and the subadditivity of the Frobenius norm,  we have that under condition~\eqref{pert_cond}
\begin{equation}
\begin{split}
\norm{B\bar{P} - \bar{P}\widetilde{T}}_F & \leq \norm{B\widetilde{P} - \widetilde{P}\widetilde{T}}_F + \norm{B}_F \norm{\bar{P}}_F\norm{E}_F + \norm{\widetilde{T}}_F \norm{\bar{P}}_F\norm{E}_F \\
                      & \leq \norm{B\widetilde{P} - \widetilde{P}\widetilde{T}}_F + 2 \norm{G} _F \norm{\bar{P}}_F \Big( \norm{B}_F + \norm{\widetilde{T}}_F \Big) .
\end{split}
\end{equation}
Thus the re-orthogonalization process may introduce an error bounded by $2 \norm{G}_F  \norm{\bar{P}}_F \Big( \norm{B}_F + \norm{\widetilde{T}}_F \Big)$ to the error introduced by approximation formulas \eqref{eqn:TEF} and \eqref{eq:CTEF}.

\end{document}